%% file: mndinvidentity-arxiv-revised.tex
\documentclass[preprint,12pt]{elsarticle}

\usepackage{amssymb}
\usepackage{amsthm}

\usepackage{amsmath}
\usepackage{url}

\newtheorem{theorem}{Theorem}
\newtheorem{lemma}{Lemma}
\newtheorem{cor}{Corollary}
\newtheorem{case}{Case}
\newtheorem*{definition}{Definition}
\newtheorem{example}{Example}
\newtheorem{conjecture}{Conjecture}
\newtheorem{result}{Result}

\DeclareMathOperator{\PF}{PF}

\DeclareMathOperator{\rank}{rank}
\usepackage{graphicx,color}

\DeclareMathOperator{\spl}{Split}
\DeclareRobustCommand{\En1}{\mathcal{E}_{n,1}}
\DeclareRobustCommand{\Q}{\mathcal{Q}_{n-1,n}}
\DeclareRobustCommand{\rdinv}{\operatorname{dinv}_{n-1,n}}
\DeclareRobustCommand{\dinv}{\operatorname{dinv}}
\DeclareRobustCommand{\rarea}{\operatorname{area}_{n-1,n}}
\DeclareRobustCommand{\area}{\operatorname{area}}
\newcommand{\ses}{\, = \,}

\newcommand{\coarea}{\operatorname{coarea}}
\newcommand{\arm}{\operatorname{arm}}
\newcommand{\leg}{\operatorname{leg}}
\newcommand{\pdinv}{\operatorname{pdinv}}
\newcommand{\tdinv}{\operatorname{tdinv}}
\newcommand{\mdinv}{\operatorname{maxtdinv}}
\newcommand{\ides}{\operatorname{ides}}

\journal{Discrete Mathematics}

\begin{document}

\begin{frontmatter}



\title{A simpler formula for the number of diagonal inversions of an \((m,n)\)-Parking Function and a returning Fermionic formula}


\author[Hicks]{Angela Hicks\fnref{fn2}\corref{cor2}}\fntext[fn2]{Supported by NSF grant DMS 1303761.}\cortext[cor2]{Corresponding author}
\address[Hicks]{Stanford University, Department of Mathematics, building 380, Stanford, California 94305 USA}
\ead{ashicks@stanford.edu}

\author[Leven]{Emily Leven \fnref{fn1}\corref{cor1}}\fntext[fn1]{Supported by NSF grant DGE 1144086}\cortext[cor1]{Principal corresponding author}
\address[Leven]{University of California, San Diego, Department of Mathematics, 9500 Gilman Dr. \#0112, San Diego, CA 92093 USA}
\ead{esergel@ucsd.edu}

\begin{abstract}

Recent results have placed the classical shuffle conjecture of Haglund et al. in a broader context of an infinite family of conjectures about parking functions in any rectangular lattice. The combinatorial side of the new conjectures has been defined using a complicated generalization of the dinv statistic which is composed of three parts and which is not obviously non-negative. Here we simplify the definition of dinv, prove that it is always non-negative, and give a geometric description of the statistic in the style of the classical case. We go on to show that in the $(n-1) \times n$ lattice, parking functions satisfy a fermionic formula that is similar to the one given in the classical case by Haglund and Loehr.

\end{abstract}

\begin{keyword}

rational parking functions \sep shuffle conjecture \sep dinv \sep diagonal inversions \sep fermionic formula




\MSC[2010] 05E05 \sep 05E10

\end{keyword}

\end{frontmatter}



\section{Introduction}

The classical shuffle conjecture of \cite{HHLRU} gives a well-studied combinatorial expression for the bigraded Frobenius characteristic of the diagonal harmonics:
\begin{equation}
\nabla e_n =\sum_{\PF \in \PF_n}t^{\area(PF)}q^{\dinv({PF})}F_{\ides({PF})}.
\end{equation}
For a survey of work on this conjecture, see \cite{HagBook}.  Recent results by Gorsky and Negut in \cite{GorskyNegut}, \cite{Negut2012a}, and \cite{Negut2013}, along with work by Schiffmann and Vasserot in \cite{Schiffmann2011} and \cite{Schiffmann2013}, when combined with combinatorial results of Hikita in \cite{Hakita} and Gorsky and Mazin in \cite{Mazin1} and \cite{Mazin2}, place this conjecture in the broader context of an infinite family of conjectures about collections of parking functions in an \(m\times n\) lattice for any coprime \(m,n\).
One such conjecture (with \(Q_{m,n}\) as defined below) gives that for $m$ and $n$ coprime
\begin{equation}
Q_{m,n}(-1)^n=\sum_{\PF\in \PF_{m,n}}t^{\area(PF)}q^{\dinv({PF})}F_{\ides({PF})}.
\end{equation}
It has previously been concluded that the new definitions of dinv and area, as given by Hikita, Gorsky, and Mazin, are in fact a generalization of the classical definitions (where the classical case is when $m=n+1$).  Work by Bergeron, Garsia, Leven, and Xin in \cite{CompositionalRational} and \cite{RationalOperators} has reformulated and expanded on the work of Gorsky and Negut, giving the algebraic left hand side of the above equation in terms that will appear familiar to those who know the classical conjecture and expanding it to include the noncoprime case. It is the authors' hope that this paper will similarly illuminate connections between the combinatorial side of the \((m,n)\) conjecture and the classical case.  

In particular, in the first half of the paper, we show that this new definition of dinv is positive, give a geometric interpretation mirroring the geometric interpretation of the classical case, and reduce the computation of the dinv from three terms to a classical looking term and a single ``dinv correction'' term, based solely on the shape of the path.  In the second half, we use this interpretation to prove an implication of the new \((m,n)\) conjecture in the case where \(m=n-1\) and explore the meaning of the dinv correction term in that case by giving its contribution to a new fermionic formula, reminiscent of the fermionic formula (\cite{dinvtree}) of Haglund and Loehr in the classical case.
\subsection{Summary of Results}  For the easy reference of the reader who is already familiar with the rational parking functions and their statistics, we will begin by briefly summarizing our results.  The terms will be defined and the results proved in future sections.

First, call the \emph{rational diagonal} of a north step the region of an $(m,n)$ parking function that is left of the step and between the two lines of slope $\frac{n}{m}$ that intersect the endpoints of the step, including the top line, but not the bottom. (See Figure \ref{diagonals}.)
\begin{figure}
\begin{center}
\includegraphics[width=5cm]{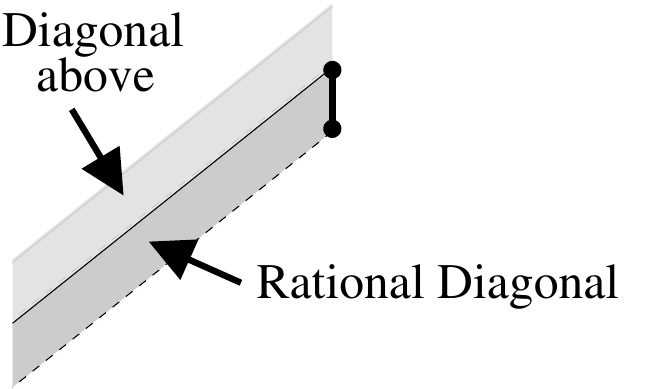}
\caption{The rational diagonal of a north step, and the diagonal above it.  Note that the southern border of each diagonal is not part of the diagonal itself.}\label{diagonals}
\end{center}
\end{figure}
\begin{result} [Theorem \ref{Result1}]  The dinv of a rational Dyck path (and thus the path dinv of a parking function) is the number of times an east step intersects the rational diagonal of a north step.
\end{result}

\noindent See Figure \ref{pathsum} for examples. 

\begin{figure}
\begin{center}
\includegraphics[width=4cm]{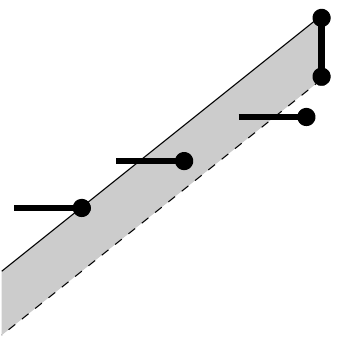}
\caption{Each east step makes a diagonal inversion with the north step.  East steps are drawn without their endpoints to indicate they may extend further.}\label{pathsum}
\end{center}
\end{figure}

Refer to the rational diagonal \emph{above} a north step $N$ as the rational diagonal of the (possibly imagined) north step directly above $N$.  Associating a labeled cell with the north step to its left, say that a labeled cell $c$ is in a diagonal $d$ if its north endpoint is in $d$.  Then $c<c'$ form a diagonal inversion if and only if:
\begin{enumerate}
\item (\emph{primary}) $c$ is in the rational diagonal of $c'$.
\item (\emph{secondary}) $c'$ is in the rational diagonal above $c$.
\end{enumerate}

\noindent A north step $N$ and and east step $E$ form a dinv correction if and only if:
\begin{enumerate}
\item (\emph{positive}) $E$ is completely contained in the rational diagonal of $N$. (This only occurs for $n<m$.)
\item (\emph{negative}) The endpoints of $E$ fall outside the rational diagonal of $N$ on opposing sides.  (This only occurs for $n>m$.)
\end{enumerate}
\begin{result}[Theorem \ref{Result2}] The dinv of a rational parking function is the total number of primary and secondary diagonal inversions, plus the number of positive dinv corrections (if $n<m$), minus the number of negative dinv corrections (if $m<n$).
\end{result}See Figure \ref{dinvsum}.
\begin{result} [Theorem \ref{Result3}] The dinv of a rational parking function is always non-negative.
\end{result}

\begin{figure}
\begin{center}
\includegraphics[width=8cm]{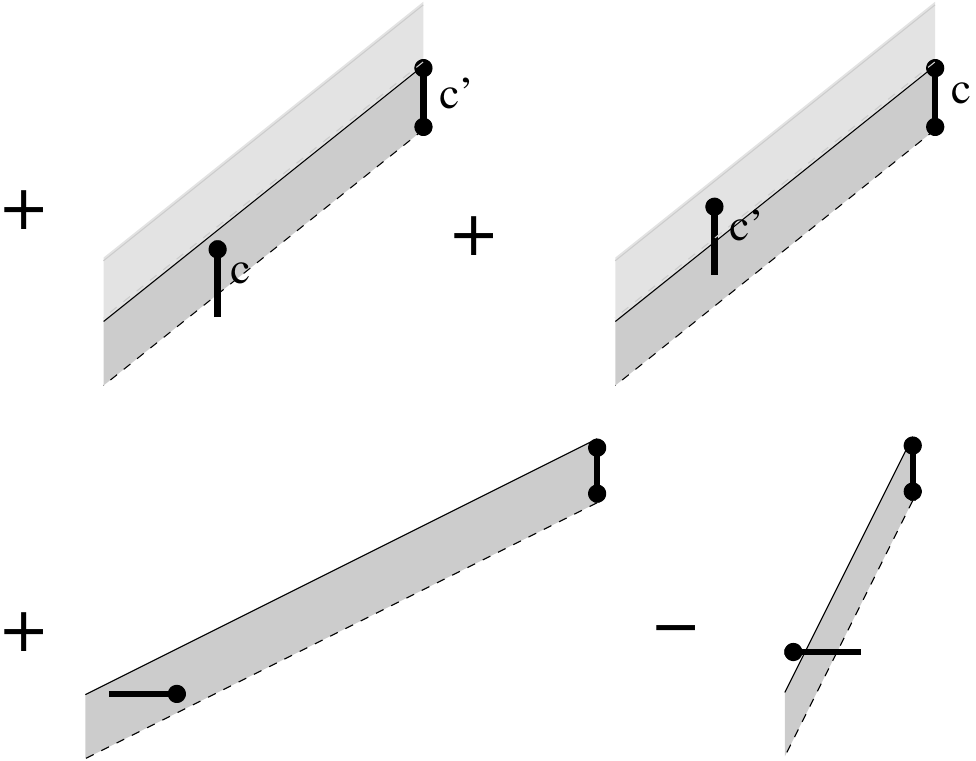}
\caption{For $c<c'$, primary and secondary inversions, as well as negative and positive dinv corrections.  The total number (with appropriate sign) give the dinv of a rational parking function.  Steps drawn without an endpoint indicate that the endpoint may be shifted sligtly to fall on the nearest marked line.}\label{dinvsum}
\end{center}
\end{figure}

Turning our attention to the case $m=n-1$, we show that:
\begin{result} [Theorem \ref{Result4}] Dinv corrections are in one to one correspondence with corners formed by an east step followed by a north step in the underlying Dyck path.
\end{result}
Define $\En1$ to be the set of classical parking functions with only a single car on the main diagonal, and $\Q$ to be the set of the rational parking functions on the $n-1$ by $n$ grid.
\begin{result} [Theorem \ref{Result5}] As is predicted algebraically by the identity $$\nabla E_{n,1}= t^{n-1}Q_{n-1,n}\, (-1)^n,$$ we have that \begin{align*}\sum_{\PF\in\En1}t^{\area(\PF)}q^{\dinv(\PF)}&F_{\ides(\PF)}
\\&=t^{n-1}\sum_{\PF\in \Q}t^{\rarea(\PF)}q^{\rdinv(\PF)}F_{\ides(\PF)}.\end{align*}  In particular, the families of parking functions in $\Q$ with the same set of cars on each classical diagonal sum to the same fermionic formula as in the classical case.
\end{result}

\section{Background}

\subsection{The classical conjecture}
Classical parking functions have been represented in a number of ways since their introduction in the 1960s.  For our purposes, it is best to consider them as represented on an $n\times n$ lattice with a ``main diagonal'' running through their center from southwest to northeast.  Construct a parking function on this lattice as follows: Add a Dyck path (a set of north and east steps, weakly above the main diagonal, also from southwest to northeast).  Call the squares directly east of north steps ``parking spaces'' and fill them with ``cars,'' integers 1 to \(n\), each occurring exactly once.  In particular, make sure that cars are strictly increasing within a single column. See Figure \ref{pfex} for an example.  

\begin{figure}[htbp]
\begin{center}
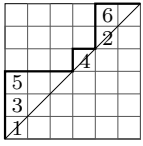
\caption{A parking function}
\label{pfex}
\end{center}
\end{figure}

\begin{definition}[area]  The area of a parking function is defined to be the number of complete squares between the Dyck path and the main diagonal.
\end{definition}

Any two cars in the same diagonal are called ``primary attacking.''  If the smaller car is to the left of the bigger car, those two cars in particular form a ``primary diagonal inversion.''  If a car $b$ is in the diagonal above the diagonal of a car $s$, and $b$ is to the left of $s$, then $s$ and $b$ are ``secondary attacking.''  Furthermore, if $s<b$, then the two cars analogously form a ``secondary diagonal inversion.''

\begin{definition}[dinv] The dinv of a parking function is the total number of primary and secondary inversion pairs in the parking function.
\end{definition}

Finally, the ``reading word'' (or just ``word'') of a parking function is the permutation found by reading the cars by diagonals, moving northeast to southwest, and working from north to south along each diagonal.  Recall that the i-descent set of a permutation is the descent set of the inverse permutation.  Then

\begin{definition}[ides] The i-descent set (ides) of a parking function is just the i-descent set of the reading word of the parking function.
\end{definition}

\begin{example} The parking function in Figure \ref{pfex} has area 4.  It has dinv 4 with the four diagonal inversions: $(1,4)$, $(1,2)$, $(2,3)$ and $(3,6)$.  Its reading word is $(5,6,3,2,4,1)$ with resulting i-descent set $\{1,2,4\}$. 
\end{example}

If we recall that $\nabla$ is the linear operator defined on the modified Macdonald polynomials $\widetilde{H}_\mu[x;q,t]$ by
\[\nabla\widetilde{H}_\mu[x;q,t]=t^{\sum_{c \in \mu} leg(c)}q^{\sum_{c \in \mu} arm(c)}\widetilde{H}_\mu[x;q,t],\] $e_n$ is the elementary symmetric function, and $F_S$ is the Gessel quasisymmetric function of degree $n$ indexed by the set $S \subseteq \{1,2,\dots,n-1\}$, then we have sufficient background to understand the shuffle conjecture as first introduced in \cite{HHLRU}:

\begin{conjecture} \[\nabla e_n =\sum_{\PF \in \PF_n}t^{\area(PF)}q^{\dinv({PF})}F_{\ides({PF})}.\]
\end{conjecture}

\subsection{The algebraic side of the $(m,n)$ shuffle conjecture}
Next, to keep the paper self-contained, we give a short summary of the algebraic side of the conjecture, choosing to mimic the notation of \cite{CompositionalRational}, a much more complete reference.  The starting point for our definitions is a pair of classically defined operators (where $[\,]$ indicates plethystic notation and $M=(1-t)(1-q)$):
\begin{align}
D_k F[X]=\, &F\left[X+\frac{M}{z}\right]\left.\sum_{i\geq 0} (-z)^ie_i[X]\right|_{z^k}\\
D_k^* F[X]=\, &F\left[X-\frac{(1-1/t)(1-1/q)}{z}\right]\left.\sum_{i\geq 0} z^ih_i[X]\right|_{z^k}
\end{align}
for any symmetric function $F$.
We use these operators to recursively build an operator $Q_{m,n}$ for each coprime pair $(m,n)$.  If $(a,b)$ is the closest lattice point below the line from $(0,0)$ to $(m,n)$ and $(m,n)=(a,b)+(c,d)$, define \[\spl(m,n)=(a,b)+(c,d).\] Then we can recursively define the relevant operators as:
\[Q_{m,n}=\begin{cases}\frac{1}{M}[Q_{c,d},Q_{a,b}]& \mbox{ if } m>1 \mbox{ and }\spl(m,n)=(a,b)+(c,d)\\
D_n &\text{ if }m=1
\end{cases}\]

\begin{example} $\spl (2,3)= (1,2)+(1,1)$, so $$Q_{2,3}=\frac{1}{M}[Q_{1,1},Q_{1,2}]=\frac{1}{M}(D_2 D_1-D_1 D_2)$$
\end{example}
In particular, we can also use the same procedure to define \(Q_{km,kn}\) (i.e.\ for no comprime pairs.)  In this case there are $k$ closest lattice points, but we refer the reader to \cite{CompositionalRational} for a proof that any choice gives the same operator.

In the coprime case, the $(m,n)$ shuffle conjecture is, as stated above:
\[Q_{m,n}(-1)^n=\sum_{\PF\in \PF_{m,n}}t^{\area(PF)}q^{\dinv({PF})}F_{\ides({PF})}.\]
The coprime shuffle conjecture, however, does not naturally extend to the non-coprime case.   To state the extension, we need to expand the definition of $Q_{m,n}$ to say that $$\textstyle Q_{0,n}= \frac{qt}{qt-1} \underline{h}_n \left[ \frac{1-qt}{qt} X \right]\hspace{.2in}\left(\hbox{multiplication by $ \frac{qt}{qt-1} h_n \left[ \frac{1-qt}{qt} X \right]$}.\right)$$   Then the $Q_{0,n}$ form a multiplicative basis for the symmetric functions and we can express the operator ``multiplication by $e_k$'' ($\underline{e
}_k$) uniquely as:
\[\underline{e}_k=\sum_{\lambda\vdash k}c_\lambda(q,t)\prod_i Q_{0,\lambda_i}.\]
Then for $(m,n)$ coprime and $k>1$, define 
\[{e_{km,kn}}=\sum_{\lambda\vdash k}c_\lambda(q,t)\prod_i Q_{\lambda_i m,\lambda_i n}.\]
Finally, the fully general version of the $(m,n)$ shuffle conjecture is that:
\[e_{km,kn}\,(-1)^{k(n+1)}=\sum_{\PF\in \PF_{km,kn}}t^{\area(PF)}q^{\dinv({PF})}F_{\ides({PF})}\]
for suitable extensions of the statistics on the combinatorial side.

\subsection{The combinatorial side of the $(m,n)$ shuffle conjecture} \label{sec:intro}

Let \( m \) and \( n \) be positive. Extending the definitions of Hikita, Gorsky, and Mazin, parking functions and their statistics can be defined for the \( m \times n \) lattice as follows. (We refer to the objects defined here as \emph{rational} in contrast to the classical case.)

\begin{figure}[htbp]
\begin{center}
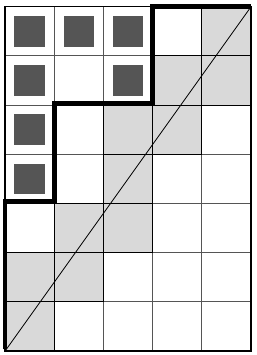
\caption{A $(5,7)$-Dyck path}
\label{57Dyck}
\end{center}
\end{figure}

An \((m,n)\)-Dyck path is a path in the \( m \times n \) lattice which proceeds by north and east steps from \((0,0)\) to \((m,n)\) and which always remains weakly above the main diagonal \( y = \frac{n}{m} x \). For example, see Figure \ref{57Dyck}.

\begin{definition}[area]
The number of full cells between an \((m,n)\)-Dyck path \( \Pi \) and the main diagonal is denoted \( \area(\Pi) \). 
\end{definition}

\begin{example}
The $(5,7)$-Dyck path in Figure \ref{57Dyck} has \( \area \) 4. 
\end{example}
When there is no danger of confusion, we will omit the $(m,n)$ prefix and just refer to the Dyck path.  Similarly we will sometimes refer to parking functions rather than $(m,n)$-parking functions.

Notice that the collection of cells above a Dyck path \( \Pi \) forms an english Ferrers diagram \( \lambda(\Pi) \). For any cell \( c \in \lambda(\Pi) \), let \( leg(c) \) and \( arm(c) \) denote the number of cells in \( \lambda(\Pi) \) which are strictly south or strictly east of \( c \), respectively. 

\begin{definition}[pdinv]
The \( \dinv \) of a Dyck path \( \Pi \) is given by
\begin{equation} \label{pdinv}
\pdinv(\Pi) \,=\, \sum_{c \in \lambda(\Pi)} \chi \left( \frac{\arm(c)}{\leg(c)+1} \leq \frac{m}{n} < \frac{\arm(c)+1}{\leg(c)} \right).
\end{equation}
 If a parking function \( \PF \) is supported by the Dyck path \( \Pi \), we will write \( \pdinv(\Pi) = \pdinv(\PF) \) and refer to this as the ``\emph{path} \( \dinv \)'' of \( \PF \).
\end{definition}
Note that here the convention is that division by 0 gives infinity.  In this formula we do not need to evaluate $\frac{0}{0}$, but later we will use the convention that this is 0.
\begin{example} The Dyck path in Figure \ref{57Dyck} has \( \lambda = (3,3,1,1) \). The 7 cells which contribute to its \( \dinv \) are marked.
\end{example}

As in the classical case, an \((m,n)\)-parking function \( \PF \) is obtained by labeling the cells east of and adjacent to north steps of an \((m,n)\)-Dyck path with \(1,2,\dots,n \) in a column-increasing way. This underlying Dyck path will be denoted \( \Pi(\PF) \). As in the classical case, the cells directly east of north steps will be called ``parking spaces'' and we will fill them with ``cars.'' For brevity, let \( \area(\PF) = \area(\Pi(\PF))\).

\begin{figure}
\begin{center}
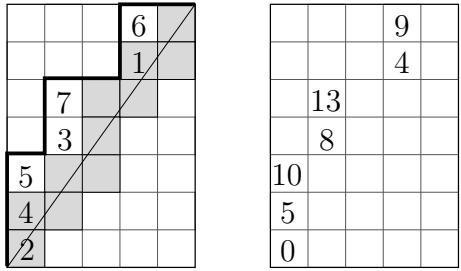
\caption{A $(5,7)$-parking function and the ranks of its cars.}
\label{57Ranks}
\end{center}
\end{figure}

The definition of \( \dinv(\PF) \) and \( \sigma(\PF) \) will be more complex. In the classical case, these statistics are related to the idea of grouping cars by diagonals. Here, we extend this notion by using the function \( \rank'(x,y) = my-nx \). A point which is on the main diagonal has \( \rank' = 0 \). In fact, \( \rank' \) orders points according to their distance from the main diagonal. In the classical case, cars are ordered according to their distance from the main diagonal so that cars further to the right are larger within a single diagonal. Analogously, we want to ``break ties'' made by \( \rank' \) so that points further right are larger. To do this, we set \( \rank(x,y) = my-nx + \lfloor \frac{x \, \gcd(m,n)}{m} \rfloor \). We say that the \( \rank \) of a car is the \( \rank \) of the southwest corner of its cell. Figure \ref{57Ranks} contains a \((5,7)\)-parking function and the \( \rank \)s of its cars.

In the classical case, the word of the parking function \( \sigma \) is obtained by reading the cars from highest to lowest diagonal. Here, the word \( \sigma \) of an \( (m,n) \)-parking function is obtained reading the cars from highest to lowest value of \( \rank \). In Figure \ref{57Ranks}, we have \( \sigma = 7 5 6 3 4 1 2 \). As in the classical case, we will be interested in the descent set of the inverse of this permutation, \( \ides(\PF) \). In Figure \ref{57Ranks}, we have \( \ides(\PF) = \{2,4,6\} \).

Furthermore, in the classical case, pairs of cars create diagonal inversions, or \( \dinv \), when their diagonals are not too far apart (and the cars occur in a certain order). For an \((m,n)\)-parking function \( \PF \), we set
\begin{equation}
\tdinv(\PF) = \sum_{\hbox{cars }i<j} \chi \left( \rank(i) < \rank(j) < \rank(i)+m \right).
\end{equation}
This is called the \emph{temporary} \( \dinv \), or \( \tdinv \), because it needs to be modified to obtain \( \dinv \). In Figure \ref{57Ranks}, the pairs of cars contributing to \( \tdinv \) are \( (1,3) \), \( (1,4) \), \( (3,5) \), \( (3,6) \), \( (4,6) \), \( (5,7) \), and \( (6,7) \).

The final ingredient for \( \dinv \) is
\begin{equation}
\mdinv(\Pi) = \max \{ \tdinv(\PF) : \Pi(\PF)=\Pi \}.
\end{equation}
We will also write \( \mdinv(\PF) \) for \( \mdinv( \Pi( \PF) ) \).

Finally we can define our second statistic.
\begin{definition}[dinv]
\begin{equation}
\dinv(\PF) = \tdinv(\PF) + \pdinv(\PF) - \mdinv(\PF).
\end{equation}
\end{definition}
Note that \( \pdinv(\PF) \) and \( \mdinv(\PF) \) depend only on the underlying Dyck path \( \Pi(PF) \).
\begin{figure}
\begin{center}
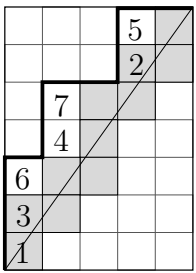
\caption{A $(5,7)$-parking function with maximal $tdinv$.}
\label{57Max}
\end{center}
\end{figure}

We show in Lemma \ref{mdinvlemma} that \( \mdinv(\PF) \) is simply the \( \tdinv \) of the parking function with underlying Dyck path \( \Pi(\PF) \) whose word is the reverse permutation \( n \, \dots \, 2 \, 1 \). The maximizing parking function corresponding to Figure \ref{57Ranks} can be found in Figure \ref{57Max}. Note that the \( \tdinv \) of this parking function is 10. Combining this with our earlier observations, we have that the \( \dinv \) of the parking function shown in Figure \ref{57Ranks} is \( 7+7-10=4 \).

\section{Another view of \( \mdinv \)}\label{sec:mdinv}

For computational purposes, and for simplicity in the arguments to follow, it is important to note that \( \mdinv(\Pi) \) can be calculated directly, without computing \( \tdinv \) for every parking function supported by \( \Pi \). Recall that a cell \( c \) of a Dyck path \( \Pi \) is called a \emph{parking space} when it is directly east of a north step - that is, when any parking function supported by \( \Pi \) ``parks'' a car in cell \( c \).

\begin{lemma} \label{mdinvlemma}
For any \( (m,n) \)-Dyck path \( \Pi \), \( \mdinv(\Pi) = \tdinv(\PF) \) where \( \PF \) is the parking function supported by \( \Pi \) with \( \sigma(\PF) = n \, \dots \, 2 \, \, 1 \). In particular,
\begin{displaymath}
\mdinv(\Pi) = \sum_{c,c'} \chi(\rank(c) < \rank(c') < \rank(c)+m)
\end{displaymath}
where the sum is over parking spaces \( c,c' \) in \( \Pi \).
\end{lemma}

\begin{proof}
Note that for any \( (m,n) \)-parking function \( \PF \) with support \( \Pi \),
\begin{align*}
\tdinv(\PF) &\ses \sum_{i<j} \chi( \rank(i) < \rank(j) < \rank(i)+m) \cr
&\ses \sum_{\rank(i) < \rank(j) < \rank(i)+m} \chi(i < j). \cr
\end{align*}
Then setting \( \sigma(\PF) = n \, \dots \, 2 \,\, 1 \) ensures that whenever \( \rank(i)<\rank(j) \), \( i<j \). In particular, if \( \rank(i)<\rank(j)<\rank(i)+m \), then \( i<j \). Hence \( \mdinv(\Pi) \) is given by the parking function with this word and, consequently, by the above formula.
\end{proof}

Note also that the line which is parallel to the main diagonal (i.e. has slope \( \frac{n}{m} \) ) and which intersects a certain point \( (x,y) \) separates the lattice into the regions of points with higher and lower ranks than \( (x,y) \). Let \( SW(c) \) and \( NW(c) \) denote the lines with slope \( \frac{n}{m} \)  which pass through the southwest and northwest corners of a cell \( c \), respectively.

\begin{figure}
\begin{center}
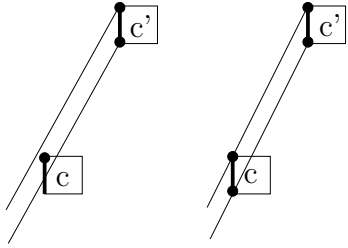
\caption{Cells contributing to primary \( \mdinv \).}
\label{primdinv}
\end{center}
\end{figure}

Note also that \( \rank(x,y)+m = \rank(x,y+1) \) and that \( \rank(c) < \rank(c') < \rank(c)+m \) if and only if \( \rank(c') < \rank(c)+m < \rank(c')+m \). Therefore when \( c' \) is further right than \( c \), we can visualize ``\( \rank(c) < \rank(c') < \rank(c)+m \)'' as one of the arrangements in Figure \ref{primdinv}. That is, this statement holds when the northwest corner of \( c \) is between \( SW(c') \) and \( NW(c') \) or when \( SW(c) = SW(c') \). When \( c \) and \( c' \) are parking spaces, such an arrangement contributes to \( \mdinv \). We will refer to this as ``primary \( \mdinv \).''

\begin{figure}
\begin{center}
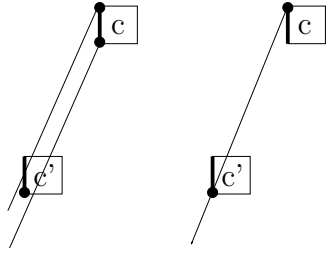
\caption{Cells contributing to secondary \( \mdinv \).}
\label{secdinv}
\end{center}
\end{figure}

On the other hand, when \( c \) is further right than \( c' \), we can view ``\( \rank(c) < \rank(c') < \rank(c)+m \)'' as one of the arrangements in Figure \ref{secdinv}. That is, the southwest corner of \( c' \) is between \( SW(c) \) and \( NW(c) \) or \(NW(c) = SW(c') \). When \( c \) and \( c' \) are parking spaces, such an arrangement also contributes to \( \mdinv \). We will refer to this as ``secondary \( \mdinv \).''


\section{The path \( \dinv \) }

We can reformulate \( \pdinv \) in a manner that is analogous to Section \ref{sec:mdinv}. Then these two reformulations will provide an inclusion of arrangements counted by \( \mdinv \) and \( \pdinv \). The direction of this inclusion will depend on whether \( m > n \) or \( m < n \). (When $m=n$, we will see that \(\pdinv = \mdinv \) as expected.) Furthermore, we will be able to characterize the leftover, giving a new formula for \( \pdinv - \mdinv \) which modifies the \( \tdinv \) of a parking function.

\begin{figure}
\begin{center}
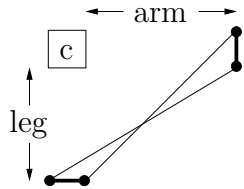
\caption{A cell contributing to path \( \dinv \).}
\label{alratios}
\end{center}
\end{figure}
\begin{theorem}\label{Result1}The dinv of a rational Dyck path (and thus the path dinv of a parking function) is the number of times an east step intersects the rational diagonal of a north step.
\end{theorem}
\begin{proof}
Consider a cell \( c \) which contributes to the \( \pdinv \) of some Dyck path \( \Pi \). We have
\begin{displaymath}
\frac{\leg(c) + 1}{\arm(c)} \geq \frac{n}{m} > \frac{\leg(c)}{\arm(c)+1}.
\end{displaymath}
The middle value is the slope of the main diagonal of \( \Pi \). We can also view the other values as slopes. Let \( N \) and \(E \) be the north and east steps of \( \Pi \) directly east and south of \( c \), respectively. Then the line connecting the north end of \( N \) and the east end of \( E \) has slope \( \frac{\leg(c)+1}{\arm(c)} \). Similarly, the line connecting the south end of \( N \) and the west end of \( E \) has slope \( \frac{\leg(c)}{\arm(c)+1} \). For example, see Figure \ref{alratios}. Then looking at the lines of slope \( \frac{n}{m} \) which intersect either end of \( N \), we have several possible configurations, all of which are enumerated in Figure \ref{aldinvtypes}.  In particular, notice these illustrations exactly give the intersections of east steps with the rational diagonal of a north steps.
\end{proof}
\begin{figure}
\begin{center}
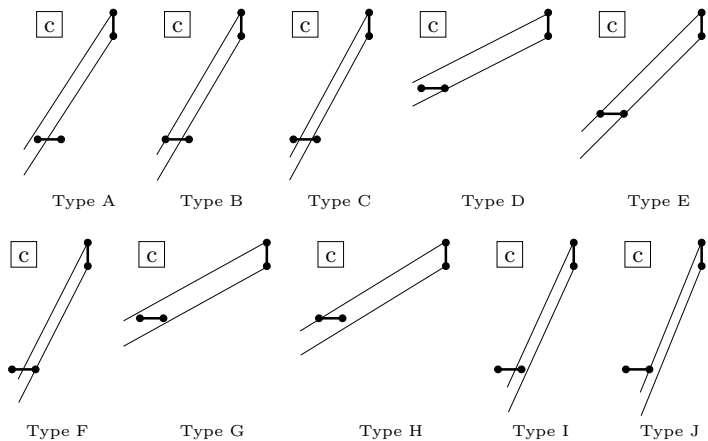
\caption{An illustration of the different types of \( \pdinv \).}
\label{aldinvtypes}
\end{center}
\end{figure}

\section{The \( \dinv \) correction}\label{sec:correction}
Next we'd like to give a simpler interpretation of the dinv correction by simultaneously analyzing path \(\pdinv\) and \(\mdinv\).  (Recall that the dinv correction is \(\pdinv-\mdinv\).) We will show that most of the configurations in Figure \ref{aldinvtypes} contribute to primary \( \mdinv \), secondary \( \mdinv \) or both. The basic idea is that if an east step crosses some line \( L \) with slope \( \frac{n}{m} \) in a Dyck path, then some north step must cross \( L \) earlier in the path. Then the parking spaces near the two north steps arising in each arrangement will contribute to \( \mdinv \). In particular, we have the following two results.

\begin{lemma} \label{primary}
Let $\Pi$ be any $(m,n)$-Dyck path. The pairs of cells which contribute to primary $\mdinv(\Pi)$ are in bijection with cells contributing to $\pdinv(\Pi)$ of Types A, B, C, D, E, and F.
\end{lemma}

\begin{figure}
\begin{center}
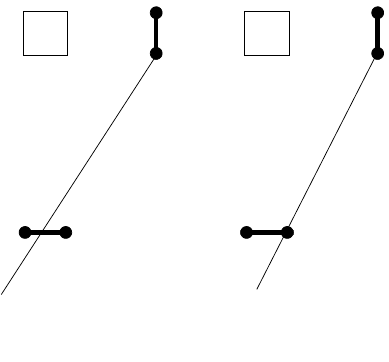
\caption{The types of \( \pdinv \) contributing to primary \( \mdinv \).}
\label{ABCDEF}
\end{center}
\end{figure}

\begin{proof}

We will group Types A, B and C and Types D, E and F as in Figure \ref{ABCDEF}. As discussed above, in each of these configurations there must be some north step crossing the sloped line which is earlier in the path than the indicated east step. Take the closest such north step. It may intersect the sloped line in its interior or on its southern endpoint. This gives four arrangements as in Figure \ref{primpdinv}. In each of these, parking spaces \( d \) and \( d' \) are in one of the arrangements of Figure \ref{primdinv}. Hence they contribute to primary \( \mdinv \).
 
\begin{figure}
\begin{center}
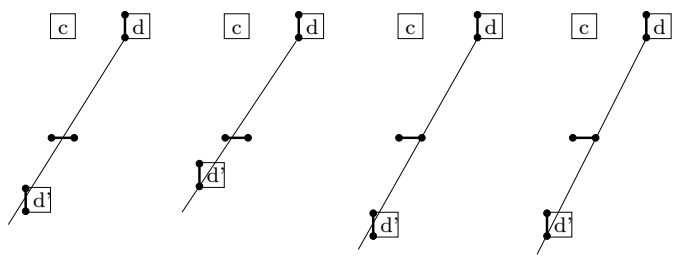
\caption{Four possible configurations corresponding to Types A-F.}
\label{primpdinv}
\end{center}
\end{figure}

Conversely, suppose parking spaces \( d \) and \( d' \) are in one of the arrangements of Figure \ref{primpdinv}. Then there is at least one east step which crosses \( SW(d) \) between the north steps of \(d \) and \(d'\). This east step may either intersect \( SW(d) \) in its interior or on its east endpoint. Choose the east step of this form which is closest to \( d' \). It is not difficult to check that the cell \(c\) north of this east step and west of \( d \) will be of Type A, B, C, D, E or F.

\end{proof}

\begin{lemma} \label{secondary}
Let $\Pi$ be any $(m,n)$-Dyck path. The pairs of cells which contribute secondary dinv to $\mdinv(\Pi)$ are in bijection with cells contributing to $\pdinv(\Pi)$ of Types C, F, I and J.
\end{lemma}

The proof of Lemma \ref{secondary} is completely analogous to that of Lemma \ref{primary}. Combining these observations, we see that \( \mdinv \) double counts \( \pdinv \) arrangements of Types C and F and does not count Types G and H. Notice, though, that Types C and F correspond simply to those cells \( c \in \lambda(\Pi) \) with
\begin{displaymath}
\frac{\leg(c)}{\arm(c)} \geq \frac{n}{m} > \frac{\leg(c)+1}{\arm(c)+1}
\end{displaymath}
or equivalently
\begin{displaymath}
\frac{\arm(c)}{\leg(c)} \leq \frac{m}{n} < \frac{\arm(c)+1}{\leg(c)+1}.
\end{displaymath}
Similarly, Types G and H correspond to those cells \( c \in \lambda(\Pi) \) with
\begin{displaymath}
\frac{\arm(c)}{\leg(c)} > \frac{m}{n} \geq \frac{\arm(c)+1}{\leg(c)+1}.
\end{displaymath}
Since these quantities represent inverses of slopes, we must allow division by 0 and let \( \frac{x}{0} = \infty \) for any \( x \neq 0 \). Examining the cases where \( \frac{0}{0} \) occurs, we can see that we need to set this quantity equal to \( 0 \).

Combining these lemmas with Theorem \ref{Result1}
 \begin{theorem}\label{Result2}
For any \( (m,n) \)-parking function \( \PF \) with underlying Dyck path \( \Pi \), we have
\begin{align*}
\dinv(\PF) &= \tdinv(\PF) + \sum_{c \in \lambda(\Pi)} \chi \left( \frac{\arm(c)}{\leg(c)} > \frac{m}{n} \geq \frac{\arm(c)+1}{\leg(c)+1} \right) \cr
&\hskip 24pt - \sum_{c \in \lambda(\Pi)}  \chi \left( \frac{\arm(c)}{\leg(c)} \leq \frac{m}{n} < \frac{\arm(c)+1}{\leg(c)+1} \right).
\end{align*}
Here \( \frac{0}{0} = 0 \) and \( \frac{x}{0} = \infty \) for all \( x\neq 0 \).
\end{theorem}

Furthermore, examining Figure \ref{aldinvtypes}, we can see that Types C and F can only occur when \( n > m \) while Types G and H can only occur when \( n < m \). Hence we have
\begin{cor} \label{MainCor}
Let \( \PF \) be any \( (m,n) \)-parking function with underlying Dyck path \( \Pi \). Set \( \frac{0}{0} = 0 \) and \( \frac{x}{0} = \infty \) for all \( x\neq 0 \). If \( n > m \) then
\begin{displaymath}
\dinv(\PF) = \tdinv(\PF) - \sum_{c \in \lambda(\Pi)}  \chi \left( \frac{\arm(c)}{\leg(c)} \leq \frac{m}{n} < \frac{\arm(c)+1}{\leg(c)+1} \right).
\end{displaymath}
If \( n = m \) then
\begin{displaymath}
\dinv(\PF) = \tdinv(\PF).
\end{displaymath}
Finally, if \( n < m \) then
\begin{displaymath}
\dinv(\PF) = \tdinv(\PF) + \sum_{c \in \lambda(\Pi)} \chi \left( \frac{\arm(c)}{\leg(c)} > \frac{m}{n} \geq \frac{\arm(c)+1}{\leg(c)+1} \right).
\end{displaymath}
\end{cor}

\begin{theorem}\label{Result3}
For any \( (m,n) \)-parking function \(\PF\), \( \dinv(\PF) \geq 0\).
\end{theorem}

\begin{proof}
This follows from Corollary \ref{MainCor} when \( m \geq n\). Let \( m < n \) and let \( PF \) be any \( (m,n) \)-parking function. We will show that for every instance of a cell \(c \in \lambda(\Pi(PF)))\) with \( \frac{\arm(c)}{\leg(c)} \leq \frac{m}{n} < \frac{\arm(c)+1}{\leg(c)+1} \), there is a triple of cars in \(\PF\) such that two of these cars must contribute to \( \tdinv(\PF) \).

Given such a cell \(c\), let \(N\) be the north step directly east of \(c\) and let \(E\) be the east step directly south of \(c\). Recall that since  \( \frac{\arm(c)}{\leg(c)} \leq \frac{m}{n} < \frac{\arm(c)+1}{\leg(c)+1} \), the lines corresponding to the ends of \(N\) must pass through \(E\) (including its east end).

Consider the closest north step \(N'\) which crosses the line corresponding to the south end of \(N\). Since \(N'\) is the closest north step, it cannot be followed by an east step other than \(E\) (see Figure \ref{Nclosest}). Therefore there is either a north step \(N''\) following \(N'\) in the Dyck path or \(E\) follows \(N'\). The latter situation is impossible, as seen in Figure \ref{closeE}.

\begin{figure}
\begin{center}
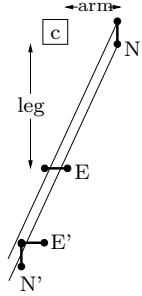
\caption{If \(N'\) is followed by an east step \(E'\), there must be a closer north step between \(E\) and \(E'\).}
\label{Nclosest}
\end{center}
\end{figure}

\begin{figure}
\begin{center}
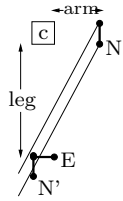
\caption{If \(E\) follows \(N'\), then the line corresponding to the north end of \(N\) cannot cross \(E\).}
\label{closeE}
\end{center}
\end{figure}

\begin{figure}
\begin{center}
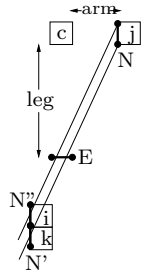
\caption{A triple of parking spaces which must contribute to 
\( \tdinv \).}
\label{forceddinv}
\end{center}
\end{figure}

 Hence there must be another north step \(N''\) directly above \(N'\) as in Figure \ref{forceddinv}. Let \(j,k,i\) be the cars adjacent to \(N,N',N''\), respectively, in \(\PF\). Clearly \(k<i\), so either \( j<k<i \), \( k<j<i \) or \( k<i<j \). The second and third cases have \( j \) and \( k \) contributing to \( \tdinv \). The first and second have \( i \) and \( j \) contributing to \( \tdinv \).

Note that not all such arrangements of \( i,j,k \) correspond to a cell \( c \) with \( \frac{\arm(c)}{\leg(c)} \leq \frac{m}{n} < \frac{\arm(c)+1}{\leg(c)+1} \). Furthermore, no such triple of cars is associated to more than one such cell \( c \). This is because above we took \( N' \) as close to \( E \) as possible. Any cells \(c\) corresponding to the triple \( i,j,k \) must be directly west of \( j \) and north of some east step crossing \( SW(j) \). Multiple such cells would imply the existence of several east steps crossing \( SW(j) \). However, only one of these can be the closest to \( N' \).

Therefore when \( m<n \), 
\[
\tdinv(\PF) \geq \sum_{c \in \lambda(\Pi)}  \chi \left( \frac{\arm(c)}{\leg(c)} \leq \frac{m}{n} < \frac{\arm(c)+1}{\leg(c)+1} \right)
\]
and consequently $\dinv(PF) \geq 0$.
\end{proof}

In summary, see Figure \ref{dinvsum} for the configurations we need to look for to compute the dinv of a rational parking function.  For each parking function, we look for what we will call, inspired by the classical case, primary diagonal inversions and secondary diagonal inversions (as in Figure \ref{primdinv} and Figure \ref{secdinv} respectively for cars $c<c'$) and dinv corrections (labeled previously by Types C, F, G, and H in Figure \ref{aldinvtypes}), which are positive or negative depending on whether $n<m$ or $n>m$.
\section{A new fermionic formula}  Since the $Q_{m,n}$ are defined by familiar, previously well-studied objects with known relations, the formulation of the $(m,n)$ shuffle conjecture has implied a number of beautiful combinatorial questions.   Inspired by one such question, the authors were led to consider the $(n-1,n)$ rational parking functions in greater detail.  We give the inspiration first, followed by a construction which ties these parking functions closely to the classical case.

Recall that as first introduced in \cite{Enk}, the $E_{n,k}$ are defined by
\[e_n\left[X\frac{1-z}{1-q}\right]=\sum_{k=1}^n \frac{(z;q)_k}{(q;q)_k} E_{n,k}\]
where $(z;q)_k =(1-z)(1-zq)\cdots (1-zq^{k-1})$.  In particular, the same paper gives the first sharpening of the shuffle conjecture, which interprets $\nabla E_{n,k}$ as the weighted sum of $n\times n$ parking functions with Dyck paths that hit the diagonal in exactly $k$ places. Specifically,
\begin{equation}
\nabla E_{n,k} =\sum_{\substack{\PF \in \PF_n \\ \PF\text{touches in $k$ places}}} t^{\area(PF)}q^{\dinv({PF})}F_{\ides({PF})}.
\end{equation}

When $z=1/q$, we have
$$
e_n\left[ X\frac{1-1/q}{1-q}\right]= \frac{1-1/q}{1-q}E_{n,1}.
$$
This gives
\begin{equation} E_{n,1}= \left(-\frac{1}{ q}\right)^{n-1}h_n[X]\label{En1}\end{equation}
Accepting the identity from \cite{RationalOperators}
$$
Q_{n-1,n}\ses -(qt)^{1-n} \nabla D_n^*\nabla^{-1},
$$
whose proof is omitted here for the sake of brevity, the following theorem conjecturally ties classical parking functions to rational parking functions.

\begin{theorem} 
\begin{equation}\label{symfunidentity}
\nabla E_{n,1}= t^{n-1}Q_{n-1,n}\, (-1)^n.
\end{equation}
\end{theorem}
\begin{proof} Using the definition of the $D_n^*$ and (\ref{En1}), we have:
$$
t^{n-1}Q_{n-1,n}\, (-1)^n= (-q )^{1-n} \nabla D_n^*\nabla^{-1}1
= (-q )^{1-n} \nabla h_n\ses \nabla E_{n,1}.
$$
\end{proof}

Since there are conjectured combinatorial interpretations of both sides of equation (\ref{symfunidentity}), it is reasonable to try to show these conjectures are consistent.  Thus, let \(\En1\) give the set of classical parking functions viewed, as originally, in the \(n\times n\) grid, with paths which touch the diagonal only where required (i.e. at the top right and bottom left).  On these objects, we may use the (relatively) straightforward classical definitions of area and dinv, which we will denote as such.  Similarly, let \(\Q\) give all parking functions in the \(({n-1}) \times n\) lattice.  On these objects, we compute the more complicated dinv and area of the rational conjectures.  In particular, we will use \(\rdinv\) and \(\rarea\) in this section to differentiate them from the classical dinv and area when there is danger of confusion.  Then the main theorem of this section is the following:
\begin{theorem} \label{Result5}
\begin{align}\label{combinatorialidentity}
\sum_{\PF\in\En1}t^{\area(\PF)}q^{\dinv(\PF)}&F_{\ides(\PF)}\\&=t^{n-1}\sum_{\PF\in \Q}t^{\rarea(\PF)}q^{\rdinv(\PF)}F_{\ides(\PF)}\nonumber
\end{align}
\end{theorem}
Noting that the ``coarea'' is the number of complete squares above a Dyck path, a first useful observation is that (\ref{combinatorialidentity}) is trivial when $q=1$ and we disregard the Gessel quasi-symmetric functions.
\begin{lemma}\label{lem:arearight}
\begin{equation*}
\sum_{\PF\in\En1}t^{\area(\PF)}=t^{n-1}\sum_{\PF\in \Q}t^{\rarea(\PF)}
\end{equation*}
\end{lemma}
\begin{proof}
In fact, any classical parking function that is in \(\En1\) and thus stays strictly above the diagonal line from \((0,0)\) to \((n,n)\) is precisely a parking function that remains above the diagonal line from \((0,0)\) to \((n-1,n)\) and thus is in \(\Q\). For \(\PF\in \En1\), \[\area(\PF)=\frac{n(n-1)}{2}-\operatorname{coarea}(PF).\]
For \(\PF\in \Q\), \[\rarea(\PF)=\frac{(n-2)(n-1)}{2}-\operatorname{coarea}(PF)\] and thus the change of area is as predicted.
\end{proof}
Furthermore, it is useful to note that the problem is not trivial from here; that is, this obvious map as used above does not in fact hold the dinv or ides of a parking function fixed.
While the dinv of a classical parking function $\PF$ is the same as the dinv of the same parking function viewed in the \(({n+1}) \times n\) grid, it is quite different from the dinv of that same parking function viewed in the \(({n-1}) \times n\) grid. Moreover, in the classical case, we read the word by diagonals, moving from northeast to southwest within a diagonal, while in the \(({n-1}) \times n\) case, ranks increase (rather than decrease, as they should when we determine the word) within a diagonal when we move from northeast to southwest.

For example, see Figure~\ref{fig:diffdinv}.  For the central, classical parking function, we notice that 2 and 3 are in the same diagonal, increasing from left to right, and thus form a secondary diagonal inversion.  Similarly, for the rightmost parking function, we notice that the top of the north step labeled by car 2 falls between  the two diagonals framing the north step beside car 3.  Thus there is a possibility of a secondary diagonal inversion between these two cars if they increase in value from left to right, as in fact they do.  Since there are no additional diagonal inversions in this parking function and no cells in the dinv correction set, as expected the \(4 \times 3\) parking function has identical statistics to the statistics for the classical parking function.  In contrast, the relative angle of the diagonals means that in the \(2\times 3\) parking function diagonals from the top and bottom of the north step beside car 3 surround the \emph{bottom} of the north step next to car 2.  Thus the two could only form a primary diagonal inversion.  Since 2 and 3 are increasing from left to right, in fact they form neither type of inversion.  In fact, there is a single secondary diagonal inversion in this parking function as well (between the 1 and 3) and a single cell in the dinv correction set, so the dinv of the first parking function is zero.  Moreover, the word of the left parking function is $(2,3,1)$, while the word of the right and center parking functions is $(3,2,1)$.
\begin{figure}
\begin{center}
 \includegraphics[width=0.75\textwidth]{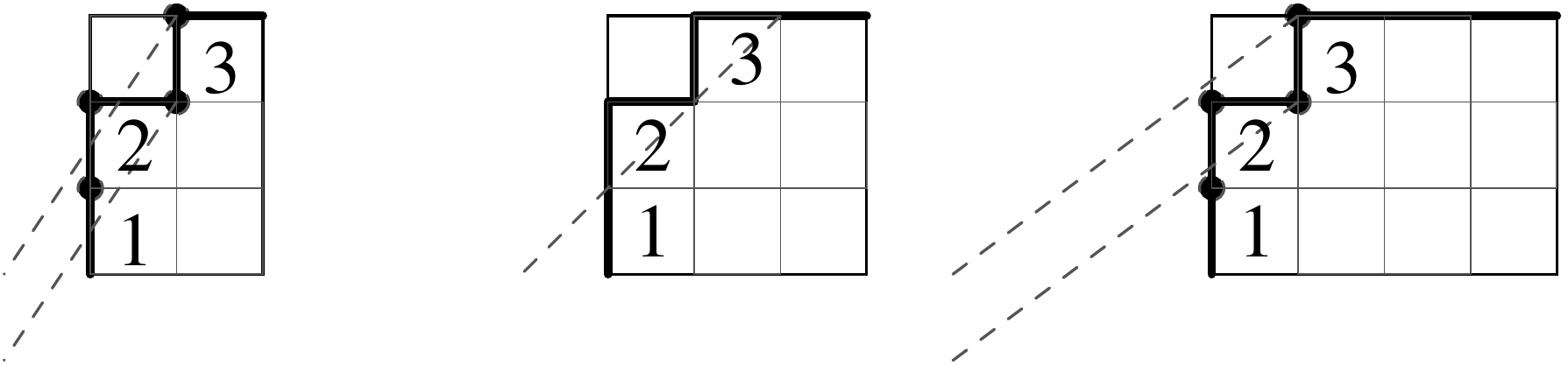}
 \caption{The dinv of each parking function (defined classically only for the middle parking function) is 0, 1, and 1 respectively.}\label{fig:diffdinv}
\end{center}
\end{figure}

In fact, it is instructive to characterize exactly what sorts of cells form a dinv or a dinv correction in the 
\((n-1)\times n\) case.
\begin{theorem} \label{Result4}
A cell in \(\PF \in \Q \) forms a dinv correction if and only if its south and east borders are formed by the path of \(\PF\).
\end{theorem}
\begin{proof}  Let \(c\) be a cell above the path such that
\[\frac{\operatorname{arm}(c)}{\operatorname{leg}(c)} \leq \frac{n-1}{n}<\frac{\operatorname{arm}(c)+1}{\operatorname{leg}(c)+1}.\] 
Since
\[\frac{\operatorname{arm}(c)}{\operatorname{leg}(c)} \leq \frac{n-1}{n}<1,\]
then \(\operatorname{arm}(c)+1<\operatorname{leg}(c)+1\) and we have 
\[\frac{n-1}{n}<\frac{\operatorname{arm}(c)+1}{\operatorname{leg}(c)+1}<\frac{n}{n}.\] 
Since \(\operatorname{leg}(c)+1\leq n\) and all arm lengths are non-negative integers, this is impossible.  The exception is the degenerate case when \(\operatorname{arm}(c)=\operatorname{leg}(c)=0\), exactly when the south and east borders of the cell are formed by the path. Hence $c$ is of this type.
\end{proof}

To finish characterizing diagonal inversions of parking functions in \(\Q\), it is useful to consider cars that are in the same `classical' diagonal, that is cars \(i\) and \(j\) such that there exists a positive integer \(k<n-1\) such that \(j\) is \(k\) cells below and \(k\) cells to the left of \(i\). Note that 
\[\rank(i) < \rank(j) + (n-1)k-nk=\rank(j)-k\]
and thus 
\[\rank(i)<\rank(j)<\rank(i)+n-1.\]
In a departure from the classical case, call these cars `primary attacking,' since they are cars that would become primary diagonal inversions if \(i<j\). Conversely, if the cells \(i\) and \(j\) are distance \(\Delta x\) in the horizontal direction and \(\Delta y\) in the vertical direction and form a primary diagonal inversion, as in Figure \ref{fig:inline} then 
\[\frac{n(\Delta x)}{n-1}-1<\Delta y<\frac{n(\Delta x)}{n-1}.\]
Since \(\Delta x\) is an nonnegative integer less than \(n-1\) and since \(\Delta y\) is an nonnegative integer less than \(n\), the leftmost inequality gives us that we have \(\Delta x\leq \Delta y\) while the right inequality gives that \(\Delta y\leq \Delta x\).  Thus only cars which are primary attacking (i.e.\ in the same classical diagonal) can create a primary diagonal inversion.

\begin{figure}
\begin{center}
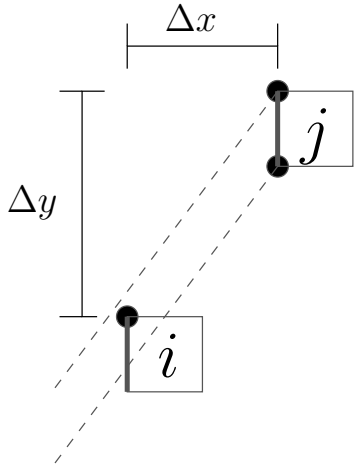
\caption{Two cars that form a primary diagonal inversion if \(i<j\)}\label{fig:inline}
\end{center}
\end{figure}

Analogously, we should refer to cars which are one (classical) diagonal apart, with the car in the upper diagonal strictly further right, as secondary attacking, and note that (by a similar argument) they form a diagonal inversion only when the larger car is in the higher diagonal.

In \cite{dinvtree}, Haglund and Loehr assigned a ``diagonal word'' to classical parking functions.  In particular, it is the unique permutation whose runs (which are in increasing order) give the cars by diagonal, starting with the highest diagonal.  See Figure \ref{fig:diagonalword} for an example.  Here we find it useful to extend this definition to cars in \(\Q\), where we say that cars are in the same diagonal if they are in the same ``classical'' diagonal.
\begin{figure}
\begin{center}
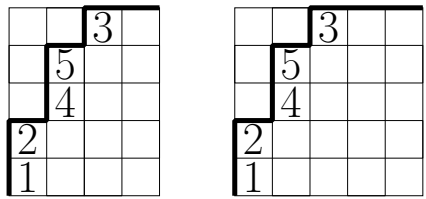
\caption{Parking functions in \(\Q\) and \(\En1\) with the same diagonal word, \((3,5,2,4,1)\).}\label{fig:diagonalword}
\end{center}
\end{figure}
In the same work, Haglund and Loehr describe a very natural recursive construction for forming the parking functions with a given diagonal word.  In particular, they describe adding cars to a parking function one at a time, starting with the last car in the diagonal word and working forward, placing each car in all possible ways, starting with the northeast-most choice and systematically working southwest.  Imagining Figure \ref{fig:dualdinvtree} with an additional empty column on the right of each parking function, we get a tree showing all the recursive choices for the diagonal word \((3,5,2,4,1)\).  In particular, Haglund and Loehr were the first to observe that at every step of this recursive procedure, choosing to move a particular car further southwest increased the diagonal inversions in the parking function by exactly one.  Thus in Figure \ref{fig:dualdinvtree}, the highest parking function has dinv 0, while the lowest has dinv 3, since three times in its construction, a car was moved past its northwest-most possible spot.  In fact, Haglund and Loehr proved that
\[\sum_{\operatorname{diag}(\PF)=\tau}t^{\area(\PF)}q^{\dinv(\PF)}=t^{\operatorname{maj}(\tau)}\prod_{i=1}^n[w_i]_q,\]
where \(w_i\) gives the number of possible positions where the \(\tau_i\)th car can be placed.  This frequently is referred to in the literature as a `fermionic formula' for the classical parking functions.
In fact, we have the following theorem that gives an interesting connection to the parking functions in \(\Q\), in particular showing that these objects satisfy the same fermionic formula.  Note that parking functions in both \(\En1\) and \(\Q\) have only a single car in the lowest diagonal, and thus their diagonal words have a final run of length one.
\begin{theorem}
For \(\tau\) with a final run of length one,
\begin{align*}\sum_{\substack{\PF \in \En1\\\operatorname{diag}(\PF)=\tau  }}t^{\area(\PF)}q^{\dinv(\PF)}&F_{\ides(\PF)}\\&=t^{n-1}\sum_{\substack{\PF \in \Q\\\operatorname{diag}(\PF)=\tau}}t^{\rarea(\PF)}q^{\rdinv(\PF)}F_{\ides(\PF)}. \end{align*}
In particular, \[\sum_{\substack{\PF \in \Q\\\operatorname{diag}(\PF)=\tau}}t^{\area(\PF)}q^{\dinv(\PF)}=t^{\operatorname{maj}(\tau)}\prod_{i=1}^n[w_i]_q,\] where the $w_i$ are as above.
\end{theorem}
\begin{proof}  We begin with almost the same recursive procedure as Haglund and Loehr.  Our point of departure is that this time we place the cars in all possible ways from southwest to northeast along a diagonal.  

Note first that while in the classical case, we read words along the diagonals from highest diagonal to lowest diagonal, working from right to left within diagonals.  For parking functions in \(\Q\), however, ranks decrease from right to left within the same classical diagonals (since we add $n-1$ and subtract $n$ as we move northeast and add $n-1$ to move north), so reading the word of such a parking function follows the same pattern.  In particular, this means that there are two types of i-descent sets in both the classical and the $(n-1)\times n$ case: `forced' i-descents created by a car $i+1$ being in a higher diagonal than car $i$ (these i-descents are forced by the diagonal word) and `optional' i-descents created when cars $i$ and $i+1$ occur in the same diagonal and are in the wrong order in the reading word.  It is clear that the forced i-descents are the same in the classical and $(n-1)\times n$ cases.  Moreover, in both cases if $i+1$ and $i$ are on the same diagonal, $i$ will be placed in the algorithm after $i+1$.  The choices of position for $i$ are the possible positions of $i+1$ plus the position northeast of $i+1$ itself.  Say that we have placed $i+1$ in each case in the $j$th possible position according to the algorithm.  (Thus in the classical case it is $j$ choices from the northeast corner, while in the $(n-1)\times n$ case it is $j$ choices from the southwest corner.)  In each case, an optional i-descent occurs if $i$ is read after $i+1$, that is when $i$ is placed in at least the $j+1$st possible position according to each respective algorithm.  Thus each algorithm produces the same i-descent set at every step.

Note that just as in the classical case the \( \coarea \) of the parking function depends only on its diagonal word. In particular within a tree, the \( \coarea \) is fixed. As observed in Lemma \ref{lem:arearight}, the \( \coarea \) of the parking function in \( \Q \) which is the highest leaf of a given tree is the same as the \( \coarea \) of the parking function in \( \En1 \) which is the lowest leaf of the corresponding tree. Thus the \( \coarea \) of all of the parking functions in \( \Q \) with a given diagonal word is the same as the \( \coarea \) of all of the parking functions in \( \En1 \) with the same diagonal word. Therefore with an appropriate change in \(\area\), it remains to observe that the \(\rdinv\) is as claimed.

  In fact, we claim that every time we recursively place a car one step further northwest, we increase the \(\rdinv\) by exactly one, starting with a choice that corresponds to no increase in dinv. Thus in Figure \ref{fig:dualdinvtree}, the lowest parking function in the second to last column corresponds to a parking function in \(\Q\) with no \(\rdinv\), while the highest in the same column gives one with 3 \(\rdinv\).  We split the proof into two lemmas below.
\end{proof}

\begin{figure}
\begin{center}
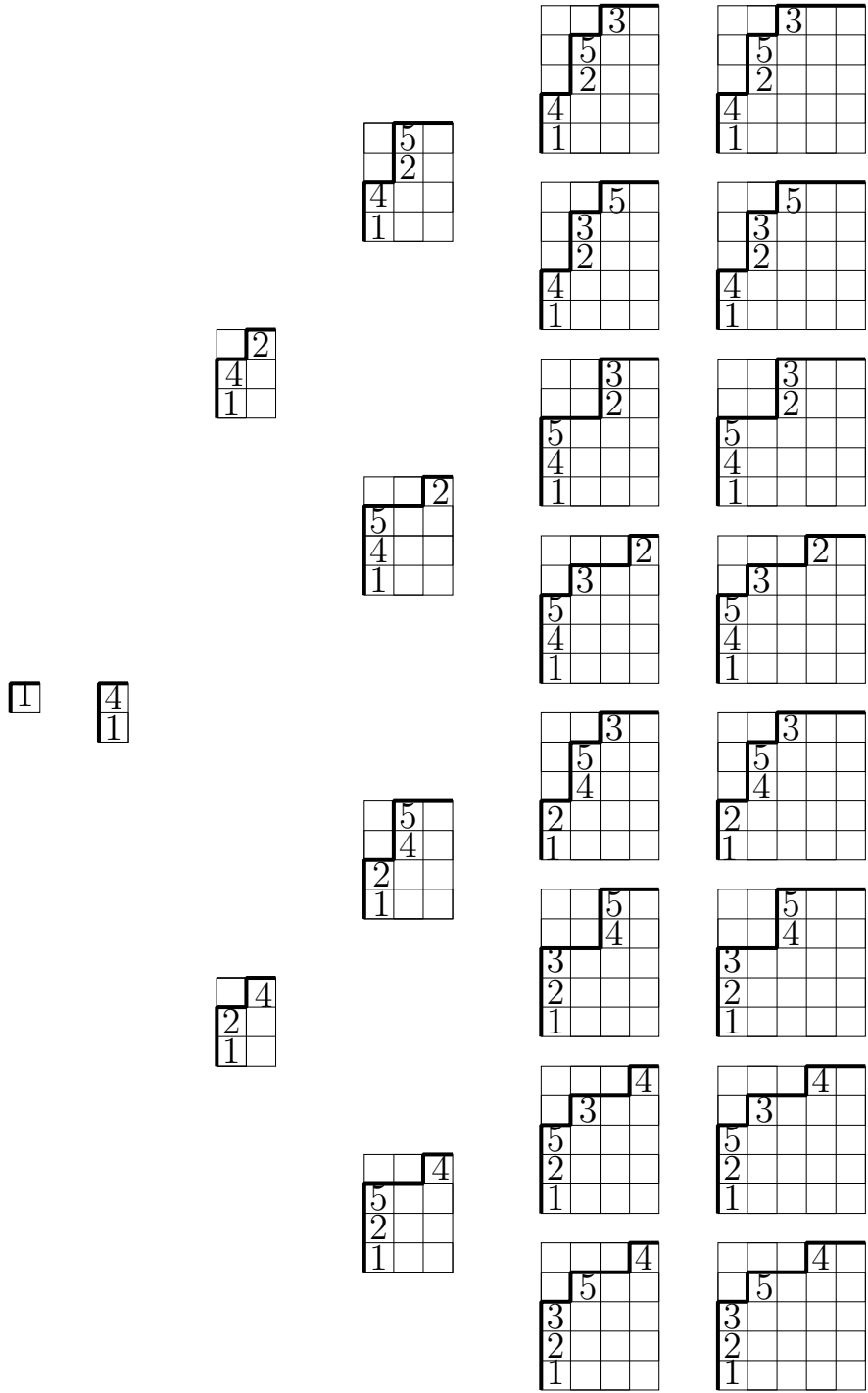
\caption{Constructing all cars in \(\Q\) with diagonal word \((3,5,2,4,1)\).}\label{fig:dualdinvtree}
\end{center}
\end{figure}

\begin{lemma}
A car placed as far southwest as possible by the procedure outlined above does not create any new dinv.
\end{lemma}

\begin{proof}
First, note that we place cars within a diagonal with a higher value first.  When we place a car \(\tau_i\), it may be placed northeast of either another (larger) car in the same diagonal or north of a smaller car in the next lowest diagonal.  
\begin{case}
\(\tau_i\) is placed above a smaller car on the next lowest diagonal
\end{case}
Any secondary diagonal inversion formed by \(\tau_i\) must include a smaller car on the next lowest diagonal, strictly southwest of \(\tau_i\), contradicting the fact that \(\tau_i\) is as far southwest as possible.  Any primary diagonal inversion formed by \(\tau_i\) must include a car in the same diagonal, and thus by necessity a larger car than \(\tau_i\).  Thus \(\tau_i\) must be to the right of a strictly larger car, again contradicting the fact that it is as far southwest as possible.  Note that adding \(\tau_i\) could in fact create a dinv correction, as placing it above a smaller car could force a (bigger) car already in the same diagonal to move one step northeast. For example in Figure \ref{fig:dualdinvtree}, we add car 3 to the bottom parking function in \(\Q\) in the diagram and thus create a new ``corner'' above the 3.  But this is exactly the only case where we create a new diagonal inversion between two already existing elements in the parking function, that is between the car below \(\tau_i\) and the bigger car to its right (i.e. as between the 2 and 5 in our example.)
\begin{case}
\(\tau_i\) is just to the northeast of a single higher car on the same diagonal
\end{case}
With an argument similar to the previous case, it is easy to see that again \(\tau_i\) can only create a secondary diagonal inversion with another car if it is not as far southwest as possible.  \(\tau_i\) can only similarly have one (bigger) car to its left in the same diagonal and thus one increase in primary diagonal inversions.  (This occurs when we add the 3 beside the 5 in the fourth parking function from the top of Figure \ref{fig:dualdinvtree}.)  Since adding \(\tau_i\) creates exactly one corner and thus one new dinv correction, the change in dinv remains zero.  Notice that in this case no car in the top diagonal is moved weakly past a car in the next diagonal, so no new diagonal inversions are created between two cars when neither of them are \(\tau_i\).
\end{proof}
\setcounter{case}{0}
\begin{figure}
\begin{center}
\includegraphics[scale=.4]{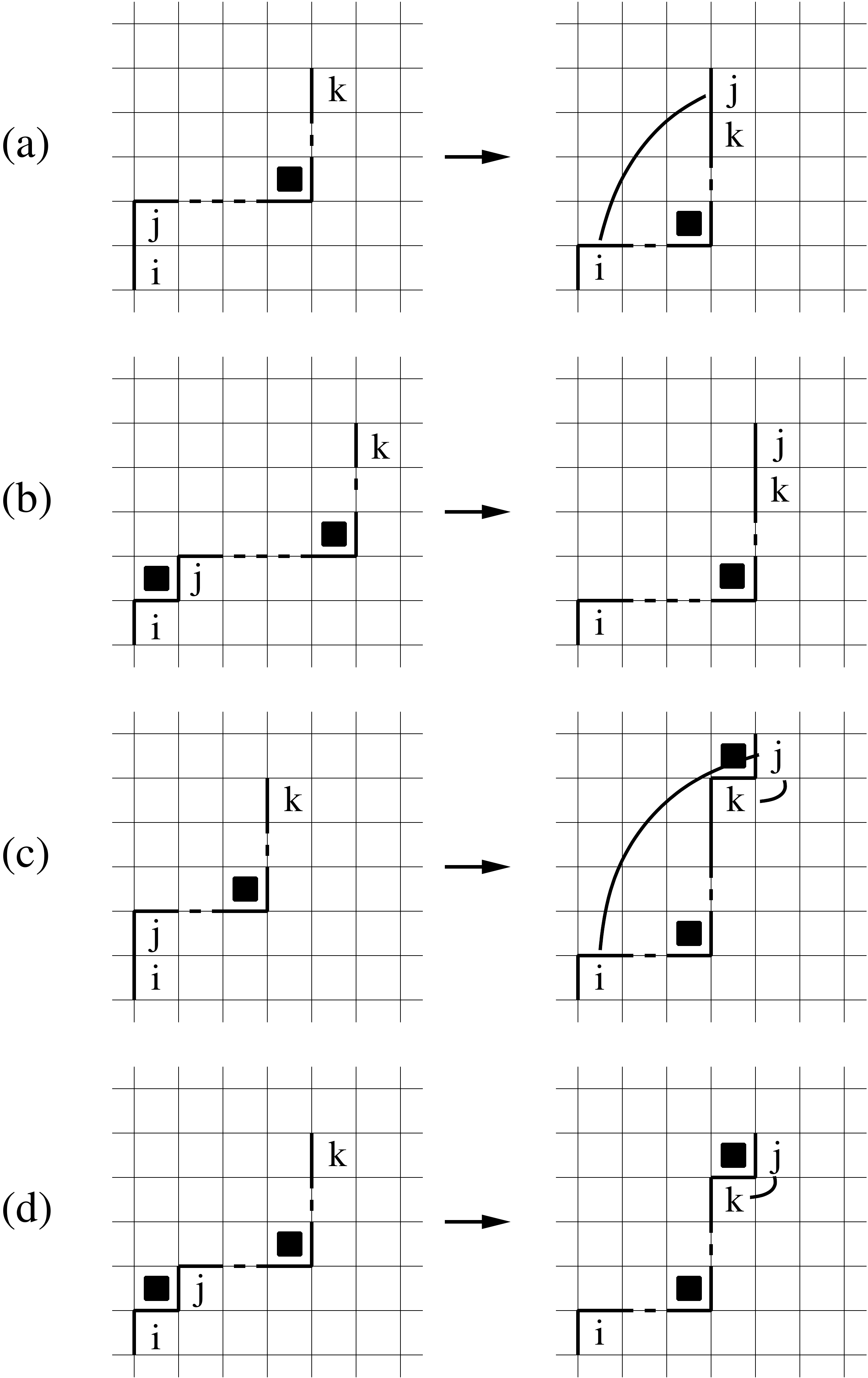} \caption{The different reasons that moving car $j$ can create an increase in dinv are highlighted by curves (to show new diagonal inversions) and squares (to show dinv correction cells).}\label{dinvproof}
\end{center}
\end{figure}
\begin{lemma}  Every time a car is moved one step further northeast by the outlined procedure, there is a corresponding increase in dinv of exactly one.
\end{lemma}
\begin{proof}  In all the below cases, let $j$ be the car moved northeast, $i$ be the first car (weakly) southwest of $j$ (i.e.\ the last car $j$ was placed beside) before the move and $k$ be the first car to the southwest after the move.  Note that $i$ and $k$ are both either bigger than $j$ and in the same diagonal (because we place bigger cars before smaller cars within a diagonal) or smaller than $j$ and in one diagonal lower.  Furthermore note that any cars between $j$ and $k$ are in a lower diagonal than $j$; if they are in the first diagonal below $j$, they must be larger than $j$ (or else we would place car $j$ on this smaller car and it would be $k$ itself). This means in particular that moving $j$ cannot create a diagonal inversion between $j$ and any intermediate car and we need only concern ourselves with diagonal inversions between the three marked cars. The below cases are diagrammed in order in Figure \ref{dinvproof}, with relevant changes in diagonal inversions marked with a curved line and relevant changes in dinv correction (a negative value in these tall parking functions) marked with a square.
\begin{case} $i$ and $k$ are in a lower diagonal than $j$ \end{case}
We must have $i<j$ and $k<j$. Then $i$ and $j$ create a new secondary diagonal inversion.
\begin{case} $i$ is in the same diagonal as $j$, while $k$ is in the next lowest diagonal \end{case}
Thus $k<j<i$.  A corner between $i$ and $j$ is no longer present after $j$ moves, resulting in a net change in diagonal inversion of one.
\begin{case} $k$ is in the same diagonal as $j$, while $i$ is in the next lowest diagonal \end{case}
In this case $i<j<k$.  Moving $j$ past $k$ creates a new corner between $k$ and $j$, a new primary dinv between $k$ and $j$, and a new secondary dinv between $i$ and $j$.
\begin{case} $i$, $j$, and $k$ are all in the same diagonal \end{case} Thus $j<i$ and $j<k$.  While a corner between $i$ and $j$ is destroyed by moving $j$, a corner between $k$ and $j$ is created.  Moreover, a primary inversion is created between $k$ and $j$.
\end{proof}

\section{Final remarks}
The results of this paper highlight several unanswered combinatorial questions related to the rational shuffle conjectures. For example, is there a ``fermionic'' type formula for all rational cases?  Here, the natural analogue of a diagonal word is not quite so clear, as we have no natural definition of the set of all cars contained within a diagonal. 

Our final result shows that the classical shuffle conjecture is consistent with the $(n-1,n)$ rational shuffle conjecture.  Since there are a plethora of relations between $Q_{kn,km}$ for all choices of $k$, $n$, and $m$, it begs the question: ``What similar combinatorial relations can be shown consistent with the conjectures?'' Perhaps the relations between the various conjectures will eventually reveal a way to prove the entire family of conjectures true by some kind of induction.

\bibliographystyle{elsarticle-num} 
\bibliography{mndinvidentity-arxiv-revised}

\end{document}

%% file: pfex.pdf_t
\begin{picture}(0,0)%
\includegraphics{pfex.pdf}%
\end{picture}%
\setlength{\unitlength}{1243sp}%
\begingroup\makeatletter\ifx\SetFigFont\undefined%
\gdef\SetFigFont#1#2#3#4#5{%
  \reset@font\fontsize{#1}{#2pt}%
  \fontfamily{#3}\fontseries{#4}\fontshape{#5}%
  \selectfont}%
\fi\endgroup%
\begin{picture}(2788,2788)(4457,-3705)
\put(4636,-3571){\makebox(0,0)[lb]{\smash{{\SetFigFont{8}{9.6}{\familydefault}{\mddefault}{\updefault}{\color[rgb]{0,0,0}1}%
}}}}
\put(6436,-1786){\makebox(0,0)[lb]{\smash{{\SetFigFont{8}{9.6}{\familydefault}{\mddefault}{\updefault}{\color[rgb]{0,0,0}2}%
}}}}
\put(5986,-2251){\makebox(0,0)[lb]{\smash{{\SetFigFont{8}{9.6}{\familydefault}{\mddefault}{\updefault}{\color[rgb]{0,0,0}4}%
}}}}
\put(4636,-2671){\makebox(0,0)[lb]{\smash{{\SetFigFont{8}{9.6}{\familydefault}{\mddefault}{\updefault}{\color[rgb]{0,0,0}5}%
}}}}
\put(4636,-3121){\makebox(0,0)[lb]{\smash{{\SetFigFont{8}{9.6}{\familydefault}{\mddefault}{\updefault}{\color[rgb]{0,0,0}3}%
}}}}
\put(6436,-1321){\makebox(0,0)[lb]{\smash{{\SetFigFont{8}{9.6}{\familydefault}{\mddefault}{\updefault}{\color[rgb]{0,0,0}6}%
}}}}
\end{picture}%

%% file: 57Dyckx.pdf_t
\begin{picture}(0,0)%
\includegraphics{57Dyckx.pdf}%
\end{picture}%
\setlength{\unitlength}{2072sp}%
\begingroup\makeatletter\ifx\SetFigFont\undefined%
\gdef\SetFigFont#1#2#3#4#5{%
  \reset@font\fontsize{#1}{#2pt}%
  \fontfamily{#3}\fontseries{#4}\fontshape{#5}%
  \selectfont}%
\fi\endgroup%
\begin{picture}(2338,3223)(4457,-4140)
\end{picture}%

%% file: 57Ranksx.pdf_t
\begin{picture}(0,0)%
\includegraphics{57Ranksx.pdf}%
\end{picture}%
\setlength{\unitlength}{2072sp}%
\begingroup\makeatletter\ifx\SetFigFont\undefined%
\gdef\SetFigFont#1#2#3#4#5{%
  \reset@font\fontsize{#1}{#2pt}%
  \fontfamily{#3}\fontseries{#4}\fontshape{#5}%
  \selectfont}%
\fi\endgroup%
\begin{picture}(5466,3223)(4457,-4140)
\put(4651,-4021){\makebox(0,0)[lb]{\smash{{\SetFigFont{12}{14.4}{\familydefault}{\mddefault}{\updefault}{\color[rgb]{0,0,0}2}%
}}}}
\put(4636,-3586){\makebox(0,0)[lb]{\smash{{\SetFigFont{12}{14.4}{\familydefault}{\mddefault}{\updefault}{\color[rgb]{0,0,0}4}%
}}}}
\put(4636,-3151){\makebox(0,0)[lb]{\smash{{\SetFigFont{12}{14.4}{\familydefault}{\mddefault}{\updefault}{\color[rgb]{0,0,0}5}%
}}}}
\put(5986,-1786){\makebox(0,0)[lb]{\smash{{\SetFigFont{12}{14.4}{\familydefault}{\mddefault}{\updefault}{\color[rgb]{0,0,0}1}%
}}}}
\put(5986,-1336){\makebox(0,0)[lb]{\smash{{\SetFigFont{12}{14.4}{\familydefault}{\mddefault}{\updefault}{\color[rgb]{0,0,0}6}%
}}}}
\put(5101,-2701){\makebox(0,0)[lb]{\smash{{\SetFigFont{12}{14.4}{\familydefault}{\mddefault}{\updefault}{\color[rgb]{0,0,0}3}%
}}}}
\put(5086,-2266){\makebox(0,0)[lb]{\smash{{\SetFigFont{12}{14.4}{\familydefault}{\mddefault}{\updefault}{\color[rgb]{0,0,0}7}%
}}}}
\put(7666,-3121){\makebox(0,0)[lb]{\smash{{\SetFigFont{12}{14.4}{\familydefault}{\mddefault}{\updefault}{\color[rgb]{0,0,0}10}%
}}}}
\put(7771,-3586){\makebox(0,0)[lb]{\smash{{\SetFigFont{12}{14.4}{\familydefault}{\mddefault}{\updefault}{\color[rgb]{0,0,0}5}%
}}}}
\put(7771,-4036){\makebox(0,0)[lb]{\smash{{\SetFigFont{12}{14.4}{\familydefault}{\mddefault}{\updefault}{\color[rgb]{0,0,0}0}%
}}}}
\put(9136,-1336){\makebox(0,0)[lb]{\smash{{\SetFigFont{12}{14.4}{\familydefault}{\mddefault}{\updefault}{\color[rgb]{0,0,0}9}%
}}}}
\put(9136,-1786){\makebox(0,0)[lb]{\smash{{\SetFigFont{12}{14.4}{\familydefault}{\mddefault}{\updefault}{\color[rgb]{0,0,0}4}%
}}}}
\put(8236,-2701){\makebox(0,0)[lb]{\smash{{\SetFigFont{12}{14.4}{\familydefault}{\mddefault}{\updefault}{\color[rgb]{0,0,0}8}%
}}}}
\put(8131,-2236){\makebox(0,0)[lb]{\smash{{\SetFigFont{12}{14.4}{\familydefault}{\mddefault}{\updefault}{\color[rgb]{0,0,0}13}%
}}}}
\end{picture}%

%% file: 57Maxx.pdf_t
\begin{picture}(0,0)%
\includegraphics{57Maxx.pdf}%
\end{picture}%
\setlength{\unitlength}{2072sp}%
\begingroup\makeatletter\ifx\SetFigFont\undefined%
\gdef\SetFigFont#1#2#3#4#5{%
  \reset@font\fontsize{#1}{#2pt}%
  \fontfamily{#3}\fontseries{#4}\fontshape{#5}%
  \selectfont}%
\fi\endgroup%
\begin{picture}(2338,3223)(4457,-4155)
\put(5086,-2266){\makebox(0,0)[lb]{\smash{{\SetFigFont{12}{14.4}{\familydefault}{\mddefault}{\updefault}{\color[rgb]{0,0,0}7}%
}}}}
\put(4621,-3136){\makebox(0,0)[lb]{\smash{{\SetFigFont{12}{14.4}{\familydefault}{\mddefault}{\updefault}{\color[rgb]{0,0,0}6}%
}}}}
\put(4636,-3586){\makebox(0,0)[lb]{\smash{{\SetFigFont{12}{14.4}{\familydefault}{\mddefault}{\updefault}{\color[rgb]{0,0,0}3}%
}}}}
\put(4621,-4021){\makebox(0,0)[lb]{\smash{{\SetFigFont{12}{14.4}{\familydefault}{\mddefault}{\updefault}{\color[rgb]{0,0,0}1}%
}}}}
\put(5086,-2686){\makebox(0,0)[lb]{\smash{{\SetFigFont{12}{14.4}{\familydefault}{\mddefault}{\updefault}{\color[rgb]{0,0,0}4}%
}}}}
\put(5971,-1351){\makebox(0,0)[lb]{\smash{{\SetFigFont{12}{14.4}{\familydefault}{\mddefault}{\updefault}{\color[rgb]{0,0,0}5}%
}}}}
\put(5986,-1771){\makebox(0,0)[lb]{\smash{{\SetFigFont{12}{14.4}{\familydefault}{\mddefault}{\updefault}{\color[rgb]{0,0,0}2}%
}}}}
\end{picture}%

%% file: primdinvx.pdf_t
\begin{picture}(0,0)%
\includegraphics{primdinvx.pdf}%
\end{picture}%
\setlength{\unitlength}{2072sp}%
\begingroup\makeatletter\ifx\SetFigFont\undefined%
\gdef\SetFigFont#1#2#3#4#5{%
  \reset@font\fontsize{#1}{#2pt}%
  \fontfamily{#3}\fontseries{#4}\fontshape{#5}%
  \selectfont}%
\fi\endgroup%
\begin{picture}(4085,2944)(428,-6066)
\put(1921,-3586){\makebox(0,0)[lb]{\smash{{\SetFigFont{12}{14.4}{\familydefault}{\mddefault}{\updefault}{\color[rgb]{0,0,0}c'}%
}}}}
\put(1036,-5371){\makebox(0,0)[lb]{\smash{{\SetFigFont{12}{14.4}{\familydefault}{\mddefault}{\updefault}{\color[rgb]{0,0,0}c}%
}}}}
\put(3286,-5326){\makebox(0,0)[lb]{\smash{{\SetFigFont{12}{14.4}{\familydefault}{\mddefault}{\updefault}{\color[rgb]{0,0,0}c}%
}}}}
\put(4186,-3571){\makebox(0,0)[lb]{\smash{{\SetFigFont{12}{14.4}{\familydefault}{\mddefault}{\updefault}{\color[rgb]{0,0,0}c'}%
}}}}
\end{picture}%

%% file: secdinvx.pdf_t
\begin{picture}(0,0)%
\includegraphics{secdinvx.pdf}%
\end{picture}%
\setlength{\unitlength}{2072sp}%
\begingroup\makeatletter\ifx\SetFigFont\undefined%
\gdef\SetFigFont#1#2#3#4#5{%
  \reset@font\fontsize{#1}{#2pt}%
  \fontfamily{#3}\fontseries{#4}\fontshape{#5}%
  \selectfont}%
\fi\endgroup%
\begin{picture}(3824,2996)(689,-6118)
\put(3256,-5371){\makebox(0,0)[lb]{\smash{{\SetFigFont{12}{14.4}{\familydefault}{\mddefault}{\updefault}{\color[rgb]{0,0,0}c'}%
}}}}
\put(1006,-5371){\makebox(0,0)[lb]{\smash{{\SetFigFont{12}{14.4}{\familydefault}{\mddefault}{\updefault}{\color[rgb]{0,0,0}c'}%
}}}}
\put(4171,-3541){\makebox(0,0)[lb]{\smash{{\SetFigFont{12}{14.4}{\familydefault}{\mddefault}{\updefault}{\color[rgb]{0,0,0}c}%
}}}}
\put(1936,-3541){\makebox(0,0)[lb]{\smash{{\SetFigFont{12}{14.4}{\familydefault}{\mddefault}{\updefault}{\color[rgb]{0,0,0}c}%
}}}}
\end{picture}%

%% file: alratiosx.pdf_t
\begin{picture}(0,0)%
\includegraphics{alratiosx.pdf}%
\end{picture}%
\setlength{\unitlength}{2072sp}%
\begingroup\makeatletter\ifx\SetFigFont\undefined%
\gdef\SetFigFont#1#2#3#4#5{%
  \reset@font\fontsize{#1}{#2pt}%
  \fontfamily{#3}\fontseries{#4}\fontshape{#5}%
  \selectfont}%
\fi\endgroup%
\begin{picture}(2834,2219)(1321,-2865)
\put(1936,-1306){\makebox(0,0)[lb]{\smash{{\SetFigFont{12}{14.4}{\familydefault}{\mddefault}{\updefault}{\color[rgb]{0,0,0}c}%
}}}}
\put(2821,-856){\makebox(0,0)[lb]{\smash{{\SetFigFont{12}{14.4}{\familydefault}{\mddefault}{\updefault}{\color[rgb]{0,0,0}arm}%
}}}}
\put(1336,-2176){\makebox(0,0)[lb]{\smash{{\SetFigFont{12}{14.4}{\familydefault}{\mddefault}{\updefault}{\color[rgb]{0,0,0}leg}%
}}}}
\end{picture}%

%% file: aldinvtypesx.pdf_t
\begin{picture}(0,0)%
\includegraphics{aldinvtypesx.pdf}%
\end{picture}%
\setlength{\unitlength}{1409sp}%
\begingroup\makeatletter\ifx\SetFigFont\undefined%
\gdef\SetFigFont#1#2#3#4#5{%
  \reset@font\fontsize{#1}{#2pt}%
  \fontfamily{#3}\fontseries{#4}\fontshape{#5}%
  \selectfont}%
\fi\endgroup%
\begin{picture}(12343,7650)(362,-7172)
\put(1036, 29){\makebox(0,0)[lb]{\smash{{\SetFigFont{9}{10.8}{\familydefault}{\mddefault}{\updefault}{\color[rgb]{0,0,0}c}%
}}}}
\put(3286, 29){\makebox(0,0)[lb]{\smash{{\SetFigFont{9}{10.8}{\familydefault}{\mddefault}{\updefault}{\color[rgb]{0,0,0}c}%
}}}}
\put(5536, 29){\makebox(0,0)[lb]{\smash{{\SetFigFont{9}{10.8}{\familydefault}{\mddefault}{\updefault}{\color[rgb]{0,0,0}c}%
}}}}
\put(7786, 29){\makebox(0,0)[lb]{\smash{{\SetFigFont{9}{10.8}{\familydefault}{\mddefault}{\updefault}{\color[rgb]{0,0,0}c}%
}}}}
\put(10936, 29){\makebox(0,0)[lb]{\smash{{\SetFigFont{9}{10.8}{\familydefault}{\mddefault}{\updefault}{\color[rgb]{0,0,0}c}%
}}}}
\put(586,-4021){\makebox(0,0)[lb]{\smash{{\SetFigFont{9}{10.8}{\familydefault}{\mddefault}{\updefault}{\color[rgb]{0,0,0}c}%
}}}}
\put(2836,-4021){\makebox(0,0)[lb]{\smash{{\SetFigFont{9}{10.8}{\familydefault}{\mddefault}{\updefault}{\color[rgb]{0,0,0}c}%
}}}}
\put(5986,-4021){\makebox(0,0)[lb]{\smash{{\SetFigFont{9}{10.8}{\familydefault}{\mddefault}{\updefault}{\color[rgb]{0,0,0}c}%
}}}}
\put(9136,-4021){\makebox(0,0)[lb]{\smash{{\SetFigFont{9}{10.8}{\familydefault}{\mddefault}{\updefault}{\color[rgb]{0,0,0}c}%
}}}}
\put(11386,-4021){\makebox(0,0)[lb]{\smash{{\SetFigFont{9}{10.8}{\familydefault}{\mddefault}{\updefault}{\color[rgb]{0,0,0}c}%
}}}}
\put(1171,-3031){\makebox(0,0)[lb]{\smash{{\SetFigFont{6}{7.2}{\familydefault}{\mddefault}{\updefault}{\color[rgb]{0,0,0}Type A}%
}}}}
\put(3421,-3031){\makebox(0,0)[lb]{\smash{{\SetFigFont{6}{7.2}{\familydefault}{\mddefault}{\updefault}{\color[rgb]{0,0,0}Type B}%
}}}}
\put(5671,-3031){\makebox(0,0)[lb]{\smash{{\SetFigFont{6}{7.2}{\familydefault}{\mddefault}{\updefault}{\color[rgb]{0,0,0}Type C}%
}}}}
\put(8371,-3031){\makebox(0,0)[lb]{\smash{{\SetFigFont{6}{7.2}{\familydefault}{\mddefault}{\updefault}{\color[rgb]{0,0,0}Type D}%
}}}}
\put(11296,-3031){\makebox(0,0)[lb]{\smash{{\SetFigFont{6}{7.2}{\familydefault}{\mddefault}{\updefault}{\color[rgb]{0,0,0}Type E}%
}}}}
\put(721,-7081){\makebox(0,0)[lb]{\smash{{\SetFigFont{6}{7.2}{\familydefault}{\mddefault}{\updefault}{\color[rgb]{0,0,0}Type F}%
}}}}
\put(3421,-7081){\makebox(0,0)[lb]{\smash{{\SetFigFont{6}{7.2}{\familydefault}{\mddefault}{\updefault}{\color[rgb]{0,0,0}Type G}%
}}}}
\put(6571,-7081){\makebox(0,0)[lb]{\smash{{\SetFigFont{6}{7.2}{\familydefault}{\mddefault}{\updefault}{\color[rgb]{0,0,0}Type H}%
}}}}
\put(9271,-7081){\makebox(0,0)[lb]{\smash{{\SetFigFont{6}{7.2}{\familydefault}{\mddefault}{\updefault}{\color[rgb]{0,0,0}Type I}%
}}}}
\put(11521,-7081){\makebox(0,0)[lb]{\smash{{\SetFigFont{6}{7.2}{\familydefault}{\mddefault}{\updefault}{\color[rgb]{0,0,0}Type J}%
}}}}
\end{picture}%

%% file: ABCDEFx.pdf_t
\begin{picture}(0,0)%
\includegraphics{ABCDEFx.pdf}%
\end{picture}%
\setlength{\unitlength}{1865sp}%
\begingroup\makeatletter\ifx\SetFigFont\undefined%
\gdef\SetFigFont#1#2#3#4#5{%
  \reset@font\fontsize{#1}{#2pt}%
  \fontfamily{#3}\fontseries{#4}\fontshape{#5}%
  \selectfont}%
\fi\endgroup%
\begin{picture}(3940,3645)(665,-3167)
\put(1036, 29){\makebox(0,0)[lb]{\smash{{\SetFigFont{11}{13.2}{\familydefault}{\mddefault}{\updefault}{\color[rgb]{0,0,0}c}%
}}}}
\put(3286, 29){\makebox(0,0)[lb]{\smash{{\SetFigFont{11}{13.2}{\familydefault}{\mddefault}{\updefault}{\color[rgb]{0,0,0}c}%
}}}}
\put(901,-3076){\makebox(0,0)[lb]{\smash{{\SetFigFont{8}{9.6}{\familydefault}{\mddefault}{\updefault}{\color[rgb]{0,0,0}Type ABC}%
}}}}
\put(3241,-3076){\makebox(0,0)[lb]{\smash{{\SetFigFont{8}{9.6}{\familydefault}{\mddefault}{\updefault}{\color[rgb]{0,0,0}Type DEF}%
}}}}
\end{picture}%

%% file: primpdinvx.pdf_t
\begin{picture}(0,0)%
\includegraphics{primpdinvx.pdf}%
\end{picture}%
\setlength{\unitlength}{1367sp}%
\begingroup\makeatletter\ifx\SetFigFont\undefined%
\gdef\SetFigFont#1#2#3#4#5{%
  \reset@font\fontsize{#1}{#2pt}%
  \fontfamily{#3}\fontseries{#4}\fontshape{#5}%
  \selectfont}%
\fi\endgroup%
\begin{picture}(12039,4466)(124,-3988)
\put(1036, 29){\makebox(0,0)[lb]{\smash{{\SetFigFont{8}{9.6}{\familydefault}{\mddefault}{\updefault}{\color[rgb]{0,0,0}c}%
}}}}
\put(2386, 29){\makebox(0,0)[lb]{\smash{{\SetFigFont{8}{9.6}{\familydefault}{\mddefault}{\updefault}{\color[rgb]{0,0,0}d}%
}}}}
\put(586,-3121){\makebox(0,0)[lb]{\smash{{\SetFigFont{8}{9.6}{\familydefault}{\mddefault}{\updefault}{\color[rgb]{0,0,0}d'}%
}}}}
\put(11836, 29){\makebox(0,0)[lb]{\smash{{\SetFigFont{8}{9.6}{\familydefault}{\mddefault}{\updefault}{\color[rgb]{0,0,0}d}%
}}}}
\put(10486, 29){\makebox(0,0)[lb]{\smash{{\SetFigFont{8}{9.6}{\familydefault}{\mddefault}{\updefault}{\color[rgb]{0,0,0}c}%
}}}}
\put(7336, 29){\makebox(0,0)[lb]{\smash{{\SetFigFont{8}{9.6}{\familydefault}{\mddefault}{\updefault}{\color[rgb]{0,0,0}c}%
}}}}
\put(5536, 29){\makebox(0,0)[lb]{\smash{{\SetFigFont{8}{9.6}{\familydefault}{\mddefault}{\updefault}{\color[rgb]{0,0,0}d}%
}}}}
\put(3736,-2671){\makebox(0,0)[lb]{\smash{{\SetFigFont{8}{9.6}{\familydefault}{\mddefault}{\updefault}{\color[rgb]{0,0,0}d'}%
}}}}
\put(4186, 29){\makebox(0,0)[lb]{\smash{{\SetFigFont{8}{9.6}{\familydefault}{\mddefault}{\updefault}{\color[rgb]{0,0,0}c}%
}}}}
\put(10036,-3571){\makebox(0,0)[lb]{\smash{{\SetFigFont{8}{9.6}{\familydefault}{\mddefault}{\updefault}{\color[rgb]{0,0,0}d'}%
}}}}
\put(6886,-3571){\makebox(0,0)[lb]{\smash{{\SetFigFont{8}{9.6}{\familydefault}{\mddefault}{\updefault}{\color[rgb]{0,0,0}d'}%
}}}}
\put(8776, 29){\makebox(0,0)[lb]{\smash{{\SetFigFont{8}{9.6}{\familydefault}{\mddefault}{\updefault}{\color[rgb]{0,0,0}d}%
}}}}
\end{picture}%

%% file: TypeDposE1x.pdf_t
\begin{picture}(0,0)%
\includegraphics{TypeDposE1x.pdf}%
\end{picture}%
\setlength{\unitlength}{1367sp}%
\begingroup\makeatletter\ifx\SetFigFont\undefined%
\gdef\SetFigFont#1#2#3#4#5{%
  \reset@font\fontsize{#1}{#2pt}%
  \fontfamily{#3}\fontseries{#4}\fontshape{#5}%
  \selectfont}%
\fi\endgroup%
\begin{picture}(2435,5265)(1158,-5896)
\put(1936,-1306){\makebox(0,0)[lb]{\smash{{\SetFigFont{8}{9.6}{\familydefault}{\mddefault}{\updefault}{\color[rgb]{0,0,0}c}%
}}}}
\put(1336,-2641){\makebox(0,0)[lb]{\smash{{\SetFigFont{8}{9.6}{\familydefault}{\mddefault}{\updefault}{\color[rgb]{0,0,0}leg}%
}}}}
\put(3271,-1621){\makebox(0,0)[lb]{\smash{{\SetFigFont{8}{9.6}{\familydefault}{\mddefault}{\updefault}{\color[rgb]{0,0,0}N}%
}}}}
\put(2371,-3856){\makebox(0,0)[lb]{\smash{{\SetFigFont{8}{9.6}{\familydefault}{\mddefault}{\updefault}{\color[rgb]{0,0,0}E}%
}}}}
\put(1936,-5116){\makebox(0,0)[lb]{\smash{{\SetFigFont{8}{9.6}{\familydefault}{\mddefault}{\updefault}{\color[rgb]{0,0,0}E'}%
}}}}
\put(1381,-5881){\makebox(0,0)[lb]{\smash{{\SetFigFont{8}{9.6}{\familydefault}{\mddefault}{\updefault}{\color[rgb]{0,0,0}N'}%
}}}}
\put(2431,-811){\makebox(0,0)[lb]{\smash{{\SetFigFont{7}{8.4}{\familydefault}{\mddefault}{\updefault}{\color[rgb]{0,0,0}arm}%
}}}}
\end{picture}%

%% file: TypeDposE2x.pdf_t
\begin{picture}(0,0)%
\includegraphics{TypeDposE2x.pdf}%
\end{picture}%
\setlength{\unitlength}{1243sp}%
\begingroup\makeatletter\ifx\SetFigFont\undefined%
\gdef\SetFigFont#1#2#3#4#5{%
  \reset@font\fontsize{#1}{#2pt}%
  \fontfamily{#3}\fontseries{#4}\fontshape{#5}%
  \selectfont}%
\fi\endgroup%
\begin{picture}(2272,3822)(1321,-4453)
\put(1936,-1306){\makebox(0,0)[lb]{\smash{{\SetFigFont{8}{9.6}{\familydefault}{\mddefault}{\updefault}{\color[rgb]{0,0,0}c}%
}}}}
\put(1336,-2641){\makebox(0,0)[lb]{\smash{{\SetFigFont{8}{9.6}{\familydefault}{\mddefault}{\updefault}{\color[rgb]{0,0,0}leg}%
}}}}
\put(3271,-1621){\makebox(0,0)[lb]{\smash{{\SetFigFont{8}{9.6}{\familydefault}{\mddefault}{\updefault}{\color[rgb]{0,0,0}N}%
}}}}
\put(2371,-3856){\makebox(0,0)[lb]{\smash{{\SetFigFont{8}{9.6}{\familydefault}{\mddefault}{\updefault}{\color[rgb]{0,0,0}E}%
}}}}
\put(2431,-811){\makebox(0,0)[lb]{\smash{{\SetFigFont{7}{8.4}{\familydefault}{\mddefault}{\updefault}{\color[rgb]{0,0,0}arm}%
}}}}
\put(1921,-4426){\makebox(0,0)[lb]{\smash{{\SetFigFont{8}{9.6}{\familydefault}{\mddefault}{\updefault}{\color[rgb]{0,0,0}N'}%
}}}}
\end{picture}%

%% file: TypeDposx.pdf_t
\begin{picture}(0,0)%
\includegraphics{TypeDposx.pdf}%
\end{picture}%
\setlength{\unitlength}{1243sp}%
\begingroup\makeatletter\ifx\SetFigFont\undefined%
\gdef\SetFigFont#1#2#3#4#5{%
  \reset@font\fontsize{#1}{#2pt}%
  \fontfamily{#3}\fontseries{#4}\fontshape{#5}%
  \selectfont}%
\fi\endgroup%
\begin{picture}(2637,5295)(976,-5926)
\put(1936,-1306){\makebox(0,0)[lb]{\smash{{\SetFigFont{8}{9.6}{\familydefault}{\mddefault}{\updefault}{\color[rgb]{0,0,0}c}%
}}}}
\put(1336,-2641){\makebox(0,0)[lb]{\smash{{\SetFigFont{8}{9.6}{\familydefault}{\mddefault}{\updefault}{\color[rgb]{0,0,0}leg}%
}}}}
\put(2431,-811){\makebox(0,0)[lb]{\smash{{\SetFigFont{7}{8.4}{\familydefault}{\mddefault}{\updefault}{\color[rgb]{0,0,0}arm}%
}}}}
\put(991,-4651){\makebox(0,0)[lb]{\smash{{\SetFigFont{8}{9.6}{\familydefault}{\mddefault}{\updefault}{\color[rgb]{0,0,0}N''}%
}}}}
\put(2356,-3796){\makebox(0,0)[lb]{\smash{{\SetFigFont{8}{9.6}{\familydefault}{\mddefault}{\updefault}{\color[rgb]{0,0,0}E}%
}}}}
\put(1306,-5911){\makebox(0,0)[lb]{\smash{{\SetFigFont{8}{9.6}{\familydefault}{\mddefault}{\updefault}{\color[rgb]{0,0,0}N'}%
}}}}
\put(3151,-1861){\makebox(0,0)[lb]{\smash{{\SetFigFont{8}{9.6}{\familydefault}{\mddefault}{\updefault}{\color[rgb]{0,0,0}N}%
}}}}
\put(1666,-5011){\makebox(0,0)[lb]{\smash{{\SetFigFont{8}{9.6}{\familydefault}{\mddefault}{\updefault}{\color[rgb]{0,0,0}i}%
}}}}
\put(3376,-1321){\makebox(0,0)[lb]{\smash{{\SetFigFont{8}{9.6}{\familydefault}{\mddefault}{\updefault}{\color[rgb]{0,0,0}j}%
}}}}
\put(1621,-5416){\makebox(0,0)[lb]{\smash{{\SetFigFont{8}{9.6}{\familydefault}{\mddefault}{\updefault}{\color[rgb]{0,0,0}k}%
}}}}
\end{picture}%

%% file: inline.pdf_t
\begin{picture}(0,0)%
\includegraphics{inline.pdf}%
\end{picture}%
\setlength{\unitlength}{4144sp}%
\begingroup\makeatletter\ifx\SetFigFont\undefined%
\gdef\SetFigFont#1#2#3#4#5{%
  \reset@font\fontsize{#1}{#2pt}%
  \fontfamily{#3}\fontseries{#4}\fontshape{#5}%
  \selectfont}%
\fi\endgroup%
\begin{picture}(2097,2793)(6016,-3223)
\put(7786,-1321){\makebox(0,0)[lb]{\smash{{\SetFigFont{34}{40.8}{\rmdefault}{\mddefault}{\updefault}{\color[rgb]{0.000,0.000,0.000}$j$}%
}}}}
\put(6931,-2671){\makebox(0,0)[lb]{\smash{{\SetFigFont{34}{40.8}{\rmdefault}{\mddefault}{\updefault}{\color[rgb]{0.000,0.000,0.000}$i$}%
}}}}
\put(6976,-601){\makebox(0,0)[lb]{\smash{{\SetFigFont{14}{16.8}{\rmdefault}{\mddefault}{\updefault}{\color[rgb]{0.000,0.000,0.000}$\Delta x$}%
}}}}
\put(6031,-1681){\makebox(0,0)[lb]{\smash{{\SetFigFont{14}{16.8}{\rmdefault}{\mddefault}{\updefault}{\color[rgb]{0.000,0.000,0.000}$\Delta y$}%
}}}}
\end{picture}%

%% file: diagonalword.pdf_t
\begin{picture}(0,0)%
\includegraphics{diagonalword.pdf}%
\end{picture}%
\setlength{\unitlength}{2072sp}%
\begingroup\makeatletter\ifx\SetFigFont\undefined%
\gdef\SetFigFont#1#2#3#4#5{%
  \reset@font\fontsize{#1}{#2pt}%
  \fontfamily{#3}\fontseries{#4}\fontshape{#5}%
  \selectfont}%
\fi\endgroup%
\begin{picture}(5058,2358)(16147,-1465)
\put(16876,-466){\makebox(0,0)[b]{\smash{{\SetFigFont{17}{20.4}{\familydefault}{\mddefault}{\updefault}{\color[rgb]{0,0,0}4}%
}}}}
\put(16876,-16){\makebox(0,0)[b]{\smash{{\SetFigFont{17}{20.4}{\familydefault}{\mddefault}{\updefault}{\color[rgb]{0,0,0}5}%
}}}}
\put(17326,434){\makebox(0,0)[b]{\smash{{\SetFigFont{17}{20.4}{\familydefault}{\mddefault}{\updefault}{\color[rgb]{0,0,0}3}%
}}}}
\put(16426,-1366){\makebox(0,0)[b]{\smash{{\SetFigFont{17}{20.4}{\familydefault}{\mddefault}{\updefault}{\color[rgb]{0,0,0}1}%
}}}}
\put(16426,-916){\makebox(0,0)[b]{\smash{{\SetFigFont{17}{20.4}{\familydefault}{\mddefault}{\updefault}{\color[rgb]{0,0,0}2}%
}}}}
\put(19576,-466){\makebox(0,0)[b]{\smash{{\SetFigFont{17}{20.4}{\familydefault}{\mddefault}{\updefault}{\color[rgb]{0,0,0}4}%
}}}}
\put(19576,-16){\makebox(0,0)[b]{\smash{{\SetFigFont{17}{20.4}{\familydefault}{\mddefault}{\updefault}{\color[rgb]{0,0,0}5}%
}}}}
\put(20026,434){\makebox(0,0)[b]{\smash{{\SetFigFont{17}{20.4}{\familydefault}{\mddefault}{\updefault}{\color[rgb]{0,0,0}3}%
}}}}
\put(19126,-1366){\makebox(0,0)[b]{\smash{{\SetFigFont{17}{20.4}{\familydefault}{\mddefault}{\updefault}{\color[rgb]{0,0,0}1}%
}}}}
\put(19126,-916){\makebox(0,0)[b]{\smash{{\SetFigFont{17}{20.4}{\familydefault}{\mddefault}{\updefault}{\color[rgb]{0,0,0}2}%
}}}}
\end{picture}%

%% file: dualdinvtree.pdf_t
\begin{picture}(0,0)%
\includegraphics{dualdinvtree.pdf}%
\end{picture}%
\setlength{\unitlength}{1657sp}%
\begingroup\makeatletter\ifx\SetFigFont\undefined%
\gdef\SetFigFont#1#2#3#4#5{%
  \reset@font\fontsize{#1}{#2pt}%
  \fontfamily{#3}\fontseries{#4}\fontshape{#5}%
  \selectfont}%
\fi\endgroup%
\begin{picture}(13158,21258)(8047,-9565)
\put(19171,9434){\makebox(0,0)[b]{\smash{{\SetFigFont{14}{16.8}{\familydefault}{\mddefault}{\updefault}{\color[rgb]{0,0,0}1}%
}}}}
\put(19126,9884){\makebox(0,0)[b]{\smash{{\SetFigFont{14}{16.8}{\familydefault}{\mddefault}{\updefault}{\color[rgb]{0,0,0}4}%
}}}}
\put(19621,10334){\makebox(0,0)[b]{\smash{{\SetFigFont{14}{16.8}{\familydefault}{\mddefault}{\updefault}{\color[rgb]{0,0,0}2}%
}}}}
\put(19621,10784){\makebox(0,0)[b]{\smash{{\SetFigFont{14}{16.8}{\familydefault}{\mddefault}{\updefault}{\color[rgb]{0,0,0}5}%
}}}}
\put(20071,11234){\makebox(0,0)[b]{\smash{{\SetFigFont{14}{16.8}{\familydefault}{\mddefault}{\updefault}{\color[rgb]{0,0,0}3}%
}}}}
\put(19126,6734){\makebox(0,0)[b]{\smash{{\SetFigFont{14}{16.8}{\familydefault}{\mddefault}{\updefault}{\color[rgb]{0,0,0}1}%
}}}}
\put(19126,7184){\makebox(0,0)[b]{\smash{{\SetFigFont{14}{16.8}{\familydefault}{\mddefault}{\updefault}{\color[rgb]{0,0,0}4}%
}}}}
\put(19576,7634){\makebox(0,0)[b]{\smash{{\SetFigFont{14}{16.8}{\familydefault}{\mddefault}{\updefault}{\color[rgb]{0,0,0}2}%
}}}}
\put(19576,8084){\makebox(0,0)[b]{\smash{{\SetFigFont{14}{16.8}{\familydefault}{\mddefault}{\updefault}{\color[rgb]{0,0,0}3}%
}}}}
\put(20071,8489){\makebox(0,0)[b]{\smash{{\SetFigFont{14}{16.8}{\familydefault}{\mddefault}{\updefault}{\color[rgb]{0,0,0}5}%
}}}}
\put(19171,4034){\makebox(0,0)[b]{\smash{{\SetFigFont{14}{16.8}{\familydefault}{\mddefault}{\updefault}{\color[rgb]{0,0,0}1}%
}}}}
\put(19126,4484){\makebox(0,0)[b]{\smash{{\SetFigFont{14}{16.8}{\familydefault}{\mddefault}{\updefault}{\color[rgb]{0,0,0}4}%
}}}}
\put(19126,4934){\makebox(0,0)[b]{\smash{{\SetFigFont{14}{16.8}{\familydefault}{\mddefault}{\updefault}{\color[rgb]{0,0,0}5}%
}}}}
\put(20026,5384){\makebox(0,0)[b]{\smash{{\SetFigFont{14}{16.8}{\familydefault}{\mddefault}{\updefault}{\color[rgb]{0,0,0}2}%
}}}}
\put(20026,5834){\makebox(0,0)[b]{\smash{{\SetFigFont{14}{16.8}{\familydefault}{\mddefault}{\updefault}{\color[rgb]{0,0,0}3}%
}}}}
\put(19126,1334){\makebox(0,0)[b]{\smash{{\SetFigFont{14}{16.8}{\familydefault}{\mddefault}{\updefault}{\color[rgb]{0,0,0}1}%
}}}}
\put(19126,1784){\makebox(0,0)[b]{\smash{{\SetFigFont{14}{16.8}{\familydefault}{\mddefault}{\updefault}{\color[rgb]{0,0,0}4}%
}}}}
\put(19126,2234){\makebox(0,0)[b]{\smash{{\SetFigFont{14}{16.8}{\familydefault}{\mddefault}{\updefault}{\color[rgb]{0,0,0}5}%
}}}}
\put(19576,2684){\makebox(0,0)[b]{\smash{{\SetFigFont{14}{16.8}{\familydefault}{\mddefault}{\updefault}{\color[rgb]{0,0,0}3}%
}}}}
\put(20476,3134){\makebox(0,0)[b]{\smash{{\SetFigFont{14}{16.8}{\familydefault}{\mddefault}{\updefault}{\color[rgb]{0,0,0}2}%
}}}}
\put(19081,-1411){\makebox(0,0)[b]{\smash{{\SetFigFont{14}{16.8}{\familydefault}{\mddefault}{\updefault}{\color[rgb]{0,0,0}1}%
}}}}
\put(19126,-961){\makebox(0,0)[b]{\smash{{\SetFigFont{14}{16.8}{\familydefault}{\mddefault}{\updefault}{\color[rgb]{0,0,0}2}%
}}}}
\put(19576,-466){\makebox(0,0)[b]{\smash{{\SetFigFont{14}{16.8}{\familydefault}{\mddefault}{\updefault}{\color[rgb]{0,0,0}4}%
}}}}
\put(19576,-16){\makebox(0,0)[b]{\smash{{\SetFigFont{14}{16.8}{\familydefault}{\mddefault}{\updefault}{\color[rgb]{0,0,0}5}%
}}}}
\put(20026,434){\makebox(0,0)[b]{\smash{{\SetFigFont{14}{16.8}{\familydefault}{\mddefault}{\updefault}{\color[rgb]{0,0,0}3}%
}}}}
\put(19126,-4066){\makebox(0,0)[b]{\smash{{\SetFigFont{14}{16.8}{\familydefault}{\mddefault}{\updefault}{\color[rgb]{0,0,0}1}%
}}}}
\put(19126,-3616){\makebox(0,0)[b]{\smash{{\SetFigFont{14}{16.8}{\familydefault}{\mddefault}{\updefault}{\color[rgb]{0,0,0}2}%
}}}}
\put(19126,-3121){\makebox(0,0)[b]{\smash{{\SetFigFont{14}{16.8}{\familydefault}{\mddefault}{\updefault}{\color[rgb]{0,0,0}3}%
}}}}
\put(20026,-2266){\makebox(0,0)[b]{\smash{{\SetFigFont{14}{16.8}{\familydefault}{\mddefault}{\updefault}{\color[rgb]{0,0,0}5}%
}}}}
\put(20026,-2716){\makebox(0,0)[b]{\smash{{\SetFigFont{14}{16.8}{\familydefault}{\mddefault}{\updefault}{\color[rgb]{0,0,0}4}%
}}}}
\put(19126,-6766){\makebox(0,0)[b]{\smash{{\SetFigFont{14}{16.8}{\familydefault}{\mddefault}{\updefault}{\color[rgb]{0,0,0}1}%
}}}}
\put(19126,-6316){\makebox(0,0)[b]{\smash{{\SetFigFont{14}{16.8}{\familydefault}{\mddefault}{\updefault}{\color[rgb]{0,0,0}2}%
}}}}
\put(19126,-5866){\makebox(0,0)[b]{\smash{{\SetFigFont{14}{16.8}{\familydefault}{\mddefault}{\updefault}{\color[rgb]{0,0,0}5}%
}}}}
\put(20476,-4966){\makebox(0,0)[b]{\smash{{\SetFigFont{14}{16.8}{\familydefault}{\mddefault}{\updefault}{\color[rgb]{0,0,0}4}%
}}}}
\put(19621,-5416){\makebox(0,0)[b]{\smash{{\SetFigFont{14}{16.8}{\familydefault}{\mddefault}{\updefault}{\color[rgb]{0,0,0}3}%
}}}}
\put(19126,-9421){\makebox(0,0)[b]{\smash{{\SetFigFont{14}{16.8}{\familydefault}{\mddefault}{\updefault}{\color[rgb]{0,0,0}1}%
}}}}
\put(19126,-9016){\makebox(0,0)[b]{\smash{{\SetFigFont{14}{16.8}{\familydefault}{\mddefault}{\updefault}{\color[rgb]{0,0,0}2}%
}}}}
\put(19126,-8566){\makebox(0,0)[b]{\smash{{\SetFigFont{14}{16.8}{\familydefault}{\mddefault}{\updefault}{\color[rgb]{0,0,0}3}%
}}}}
\put(19621,-8116){\makebox(0,0)[b]{\smash{{\SetFigFont{14}{16.8}{\familydefault}{\mddefault}{\updefault}{\color[rgb]{0,0,0}5}%
}}}}
\put(20476,-7666){\makebox(0,0)[b]{\smash{{\SetFigFont{14}{16.8}{\familydefault}{\mddefault}{\updefault}{\color[rgb]{0,0,0}4}%
}}}}
\put(8326,929){\makebox(0,0)[b]{\smash{{\SetFigFont{14}{16.8}{\familydefault}{\mddefault}{\updefault}{\color[rgb]{0,0,0}1}%
}}}}
\put(9676,434){\makebox(0,0)[b]{\smash{{\SetFigFont{14}{16.8}{\familydefault}{\mddefault}{\updefault}{\color[rgb]{0,0,0}1}%
}}}}
\put(9676,884){\makebox(0,0)[b]{\smash{{\SetFigFont{14}{16.8}{\familydefault}{\mddefault}{\updefault}{\color[rgb]{0,0,0}4}%
}}}}
\put(11476,-4516){\makebox(0,0)[b]{\smash{{\SetFigFont{14}{16.8}{\familydefault}{\mddefault}{\updefault}{\color[rgb]{0,0,0}1}%
}}}}
\put(11476,-4066){\makebox(0,0)[b]{\smash{{\SetFigFont{14}{16.8}{\familydefault}{\mddefault}{\updefault}{\color[rgb]{0,0,0}2}%
}}}}
\put(11971,-3661){\makebox(0,0)[b]{\smash{{\SetFigFont{14}{16.8}{\familydefault}{\mddefault}{\updefault}{\color[rgb]{0,0,0}4}%
}}}}
\put(11476,5384){\makebox(0,0)[b]{\smash{{\SetFigFont{14}{16.8}{\familydefault}{\mddefault}{\updefault}{\color[rgb]{0,0,0}1}%
}}}}
\put(11521,5834){\makebox(0,0)[b]{\smash{{\SetFigFont{14}{16.8}{\familydefault}{\mddefault}{\updefault}{\color[rgb]{0,0,0}4}%
}}}}
\put(11926,6284){\makebox(0,0)[b]{\smash{{\SetFigFont{14}{16.8}{\familydefault}{\mddefault}{\updefault}{\color[rgb]{0,0,0}2}%
}}}}
\put(13726,8084){\makebox(0,0)[b]{\smash{{\SetFigFont{14}{16.8}{\familydefault}{\mddefault}{\updefault}{\color[rgb]{0,0,0}1}%
}}}}
\put(14176,8984){\makebox(0,0)[b]{\smash{{\SetFigFont{14}{16.8}{\familydefault}{\mddefault}{\updefault}{\color[rgb]{0,0,0}2}%
}}}}
\put(13726,8579){\makebox(0,0)[b]{\smash{{\SetFigFont{14}{16.8}{\familydefault}{\mddefault}{\updefault}{\color[rgb]{0,0,0}4}%
}}}}
\put(14176,9434){\makebox(0,0)[b]{\smash{{\SetFigFont{14}{16.8}{\familydefault}{\mddefault}{\updefault}{\color[rgb]{0,0,0}5}%
}}}}
\put(13726,2684){\makebox(0,0)[b]{\smash{{\SetFigFont{14}{16.8}{\familydefault}{\mddefault}{\updefault}{\color[rgb]{0,0,0}1}%
}}}}
\put(13726,3134){\makebox(0,0)[b]{\smash{{\SetFigFont{14}{16.8}{\familydefault}{\mddefault}{\updefault}{\color[rgb]{0,0,0}4}%
}}}}
\put(13726,3584){\makebox(0,0)[b]{\smash{{\SetFigFont{14}{16.8}{\familydefault}{\mddefault}{\updefault}{\color[rgb]{0,0,0}5}%
}}}}
\put(14671,4034){\makebox(0,0)[b]{\smash{{\SetFigFont{14}{16.8}{\familydefault}{\mddefault}{\updefault}{\color[rgb]{0,0,0}2}%
}}}}
\put(13771,-2266){\makebox(0,0)[b]{\smash{{\SetFigFont{14}{16.8}{\familydefault}{\mddefault}{\updefault}{\color[rgb]{0,0,0}1}%
}}}}
\put(13771,-1816){\makebox(0,0)[b]{\smash{{\SetFigFont{14}{16.8}{\familydefault}{\mddefault}{\updefault}{\color[rgb]{0,0,0}2}%
}}}}
\put(14176,-1366){\makebox(0,0)[b]{\smash{{\SetFigFont{14}{16.8}{\familydefault}{\mddefault}{\updefault}{\color[rgb]{0,0,0}4}%
}}}}
\put(14176,-916){\makebox(0,0)[b]{\smash{{\SetFigFont{14}{16.8}{\familydefault}{\mddefault}{\updefault}{\color[rgb]{0,0,0}5}%
}}}}
\put(13726,-7666){\makebox(0,0)[b]{\smash{{\SetFigFont{14}{16.8}{\familydefault}{\mddefault}{\updefault}{\color[rgb]{0,0,0}1}%
}}}}
\put(13726,-7216){\makebox(0,0)[b]{\smash{{\SetFigFont{14}{16.8}{\familydefault}{\mddefault}{\updefault}{\color[rgb]{0,0,0}2}%
}}}}
\put(13726,-6766){\makebox(0,0)[b]{\smash{{\SetFigFont{14}{16.8}{\familydefault}{\mddefault}{\updefault}{\color[rgb]{0,0,0}5}%
}}}}
\put(14671,-6316){\makebox(0,0)[b]{\smash{{\SetFigFont{14}{16.8}{\familydefault}{\mddefault}{\updefault}{\color[rgb]{0,0,0}4}%
}}}}
\put(16471,9434){\makebox(0,0)[b]{\smash{{\SetFigFont{14}{16.8}{\familydefault}{\mddefault}{\updefault}{\color[rgb]{0,0,0}1}%
}}}}
\put(16426,9884){\makebox(0,0)[b]{\smash{{\SetFigFont{14}{16.8}{\familydefault}{\mddefault}{\updefault}{\color[rgb]{0,0,0}4}%
}}}}
\put(16921,10334){\makebox(0,0)[b]{\smash{{\SetFigFont{14}{16.8}{\familydefault}{\mddefault}{\updefault}{\color[rgb]{0,0,0}2}%
}}}}
\put(16921,10784){\makebox(0,0)[b]{\smash{{\SetFigFont{14}{16.8}{\familydefault}{\mddefault}{\updefault}{\color[rgb]{0,0,0}5}%
}}}}
\put(17371,11234){\makebox(0,0)[b]{\smash{{\SetFigFont{14}{16.8}{\familydefault}{\mddefault}{\updefault}{\color[rgb]{0,0,0}3}%
}}}}
\put(16426,6734){\makebox(0,0)[b]{\smash{{\SetFigFont{14}{16.8}{\familydefault}{\mddefault}{\updefault}{\color[rgb]{0,0,0}1}%
}}}}
\put(16426,7184){\makebox(0,0)[b]{\smash{{\SetFigFont{14}{16.8}{\familydefault}{\mddefault}{\updefault}{\color[rgb]{0,0,0}4}%
}}}}
\put(16876,7634){\makebox(0,0)[b]{\smash{{\SetFigFont{14}{16.8}{\familydefault}{\mddefault}{\updefault}{\color[rgb]{0,0,0}2}%
}}}}
\put(16876,8084){\makebox(0,0)[b]{\smash{{\SetFigFont{14}{16.8}{\familydefault}{\mddefault}{\updefault}{\color[rgb]{0,0,0}3}%
}}}}
\put(17371,8489){\makebox(0,0)[b]{\smash{{\SetFigFont{14}{16.8}{\familydefault}{\mddefault}{\updefault}{\color[rgb]{0,0,0}5}%
}}}}
\put(16471,4034){\makebox(0,0)[b]{\smash{{\SetFigFont{14}{16.8}{\familydefault}{\mddefault}{\updefault}{\color[rgb]{0,0,0}1}%
}}}}
\put(16426,4484){\makebox(0,0)[b]{\smash{{\SetFigFont{14}{16.8}{\familydefault}{\mddefault}{\updefault}{\color[rgb]{0,0,0}4}%
}}}}
\put(16426,4934){\makebox(0,0)[b]{\smash{{\SetFigFont{14}{16.8}{\familydefault}{\mddefault}{\updefault}{\color[rgb]{0,0,0}5}%
}}}}
\put(17326,5384){\makebox(0,0)[b]{\smash{{\SetFigFont{14}{16.8}{\familydefault}{\mddefault}{\updefault}{\color[rgb]{0,0,0}2}%
}}}}
\put(17326,5834){\makebox(0,0)[b]{\smash{{\SetFigFont{14}{16.8}{\familydefault}{\mddefault}{\updefault}{\color[rgb]{0,0,0}3}%
}}}}
\put(16426,1334){\makebox(0,0)[b]{\smash{{\SetFigFont{14}{16.8}{\familydefault}{\mddefault}{\updefault}{\color[rgb]{0,0,0}1}%
}}}}
\put(16426,1784){\makebox(0,0)[b]{\smash{{\SetFigFont{14}{16.8}{\familydefault}{\mddefault}{\updefault}{\color[rgb]{0,0,0}4}%
}}}}
\put(16426,2234){\makebox(0,0)[b]{\smash{{\SetFigFont{14}{16.8}{\familydefault}{\mddefault}{\updefault}{\color[rgb]{0,0,0}5}%
}}}}
\put(16876,2684){\makebox(0,0)[b]{\smash{{\SetFigFont{14}{16.8}{\familydefault}{\mddefault}{\updefault}{\color[rgb]{0,0,0}3}%
}}}}
\put(17776,3134){\makebox(0,0)[b]{\smash{{\SetFigFont{14}{16.8}{\familydefault}{\mddefault}{\updefault}{\color[rgb]{0,0,0}2}%
}}}}
\put(16381,-1411){\makebox(0,0)[b]{\smash{{\SetFigFont{14}{16.8}{\familydefault}{\mddefault}{\updefault}{\color[rgb]{0,0,0}1}%
}}}}
\put(16426,-961){\makebox(0,0)[b]{\smash{{\SetFigFont{14}{16.8}{\familydefault}{\mddefault}{\updefault}{\color[rgb]{0,0,0}2}%
}}}}
\put(16876,-466){\makebox(0,0)[b]{\smash{{\SetFigFont{14}{16.8}{\familydefault}{\mddefault}{\updefault}{\color[rgb]{0,0,0}4}%
}}}}
\put(16876,-16){\makebox(0,0)[b]{\smash{{\SetFigFont{14}{16.8}{\familydefault}{\mddefault}{\updefault}{\color[rgb]{0,0,0}5}%
}}}}
\put(17326,434){\makebox(0,0)[b]{\smash{{\SetFigFont{14}{16.8}{\familydefault}{\mddefault}{\updefault}{\color[rgb]{0,0,0}3}%
}}}}
\put(16426,-4066){\makebox(0,0)[b]{\smash{{\SetFigFont{14}{16.8}{\familydefault}{\mddefault}{\updefault}{\color[rgb]{0,0,0}1}%
}}}}
\put(16426,-3616){\makebox(0,0)[b]{\smash{{\SetFigFont{14}{16.8}{\familydefault}{\mddefault}{\updefault}{\color[rgb]{0,0,0}2}%
}}}}
\put(16426,-3121){\makebox(0,0)[b]{\smash{{\SetFigFont{14}{16.8}{\familydefault}{\mddefault}{\updefault}{\color[rgb]{0,0,0}3}%
}}}}
\put(17326,-2266){\makebox(0,0)[b]{\smash{{\SetFigFont{14}{16.8}{\familydefault}{\mddefault}{\updefault}{\color[rgb]{0,0,0}5}%
}}}}
\put(17326,-2716){\makebox(0,0)[b]{\smash{{\SetFigFont{14}{16.8}{\familydefault}{\mddefault}{\updefault}{\color[rgb]{0,0,0}4}%
}}}}
\put(16426,-6766){\makebox(0,0)[b]{\smash{{\SetFigFont{14}{16.8}{\familydefault}{\mddefault}{\updefault}{\color[rgb]{0,0,0}1}%
}}}}
\put(16426,-6316){\makebox(0,0)[b]{\smash{{\SetFigFont{14}{16.8}{\familydefault}{\mddefault}{\updefault}{\color[rgb]{0,0,0}2}%
}}}}
\put(16426,-5866){\makebox(0,0)[b]{\smash{{\SetFigFont{14}{16.8}{\familydefault}{\mddefault}{\updefault}{\color[rgb]{0,0,0}5}%
}}}}
\put(17776,-4966){\makebox(0,0)[b]{\smash{{\SetFigFont{14}{16.8}{\familydefault}{\mddefault}{\updefault}{\color[rgb]{0,0,0}4}%
}}}}
\put(16921,-5416){\makebox(0,0)[b]{\smash{{\SetFigFont{14}{16.8}{\familydefault}{\mddefault}{\updefault}{\color[rgb]{0,0,0}3}%
}}}}
\put(16426,-9421){\makebox(0,0)[b]{\smash{{\SetFigFont{14}{16.8}{\familydefault}{\mddefault}{\updefault}{\color[rgb]{0,0,0}1}%
}}}}
\put(16426,-9016){\makebox(0,0)[b]{\smash{{\SetFigFont{14}{16.8}{\familydefault}{\mddefault}{\updefault}{\color[rgb]{0,0,0}2}%
}}}}
\put(16426,-8566){\makebox(0,0)[b]{\smash{{\SetFigFont{14}{16.8}{\familydefault}{\mddefault}{\updefault}{\color[rgb]{0,0,0}3}%
}}}}
\put(16921,-8116){\makebox(0,0)[b]{\smash{{\SetFigFont{14}{16.8}{\familydefault}{\mddefault}{\updefault}{\color[rgb]{0,0,0}5}%
}}}}
\put(17776,-7666){\makebox(0,0)[b]{\smash{{\SetFigFont{14}{16.8}{\familydefault}{\mddefault}{\updefault}{\color[rgb]{0,0,0}4}%
}}}}
\end{picture}%

%% file: mndinvidentity-arxiv-revised.bbl
\begin{thebibliography}{10}
\expandafter\ifx\csname url\endcsname\relax
  \def\url#1{\texttt{#1}}\fi
\expandafter\ifx\csname urlprefix\endcsname\relax\def\urlprefix{URL }\fi
\expandafter\ifx\csname href\endcsname\relax
  \def\href#1#2{#2} \def\path#1{#1}\fi

\bibitem{HHLRU}
J.~Haglund, M.~Haiman, N.~Loehr, J.~B. Remmel, A.~Ulyanov, A combinatorial
  formula for the character of the diagonal coinvariants, Duke Math. J. 126~(2)
  (2005) 195--232.
\newblock \href {http://dx.doi.org/10.1215/S0012-7094-04-12621-1}
  {\path{doi:10.1215/S0012-7094-04-12621-1}}.

\bibitem{HagBook}
J.~Haglund, The {$q$},{$t$}-{C}atalan numbers and the space of diagonal
  harmonics, Vol.~41 of University Lecture Series, American Mathematical
  Society, Providence, RI, 2008.

\bibitem{GorskyNegut}
E.~{Gorsky}, A.~{Negut}, {Refined knot invariants and Hilbert schemes}, ArXiv
  e-prints\href {http://arxiv.org/abs/1304.3328} {\path{arXiv:1304.3328}}.

\bibitem{Negut2012a}
A.~{Negut}, {Moduli of Flags of Sheaves on P\^{}2 and their K-theory}, ArXiv
  e-prints\href {http://arxiv.org/abs/1209.4242} {\path{arXiv:1209.4242}}.

\bibitem{Negut2013}
A.~Negut, The shuffle algebra revisited, International Mathematics Research
  Notices\href {http://dx.doi.org/10.1093/imrn/rnt156}
  {\path{doi:10.1093/imrn/rnt156}}.

\bibitem{Schiffmann2011}
O.~Schiffmann, E.~Vasserot, The elliptic {H}all algebra, {C}herednik {H}ecke
  algebras and {M}acdonald polynomials, Compos. Math. 147~(1) (2011) 188--234.
\newblock \href {http://dx.doi.org/10.1112/S0010437X10004872}
  {\path{doi:10.1112/S0010437X10004872}}.

\bibitem{Schiffmann2013}
O.~Schiffmann, E.~Vasserot, The elliptic {H}all algebra and the {$K$}-theory of
  the {H}ilbert scheme of {$\Bbb A\sp 2$}, Duke Math. J. 162~(2) (2013)
  279--366.
\newblock \href {http://dx.doi.org/10.1215/00127094-1961849}
  {\path{doi:10.1215/00127094-1961849}}.

\bibitem{Hakita}
T.~{Hikita}, {Affine Springer fibers of type A and combinatorics of diagonal
  coinvariants}, ArXiv e-prints\href {http://arxiv.org/abs/1203.5878}
  {\path{arXiv:1203.5878}}.

\bibitem{Mazin1}
E.~Gorsky, M.~Mazin, {Compactified Jacobians and $q,t$-Catalan Numbers, I}, J.
  Comb. Series A 120~(1) (2013) 49--63.

\bibitem{Mazin2}
M.~Mazin, {A bijective proof of Loehr-Warrington's formulas for the statistics
  $\operatorname{ctot}_{\frac{q}{p}}$ and $\operatorname{midd}_{\frac{q}{p}}$},
  Annals of Comb.To appear.

\bibitem{CompositionalRational}
F.~Bergeron, A.~Garsia, E.~Leven, G.~Xin, {A Compositional $(km,kn)$-Shuffle
  Conjecture}, pre-print.

\bibitem{RationalOperators}
F.~Bergeron, A.~Garsia, E.~Leven, G.~Xin, {Some remarkable new plethystic
  operators in the theory of Macdonald polynomials}, pre-print.

\bibitem{dinvtree}
J.~Haglund, N.~Loehr, A conjectured combinatorial formula for the {H}ilbert
  series for diagonal harmonics, Discrete Math. 298~(1-3) (2005) 189--204.
\newblock \href {http://dx.doi.org/10.1016/j.disc.2004.01.022}
  {\path{doi:10.1016/j.disc.2004.01.022}}.

\bibitem{Enk}
A.~M. Garsia, J.~Haglund, A positivity result in the theory of {M}acdonald
  polynomials, Proc. Natl. Acad. Sci. USA 98~(8) (2001) 4313--4316.
\newblock \href {http://dx.doi.org/10.1073/pnas.071043398}
  {\path{doi:10.1073/pnas.071043398}}.

\end{thebibliography}
